\documentclass[11pt,a4paper]{amsart}

\usepackage[margin=1in]{geometry}
\usepackage{amsmath,amssymb,amsthm,mathtools}
\usepackage{enumitem}
\usepackage{booktabs}
\usepackage{hyperref}
\usepackage{graphicx}
\usepackage{multirow}

\usepackage{natbib}
\usepackage{xcolor}

\newtheorem{theorem}{Theorem}[section]
\newtheorem{proposition}[theorem]{Proposition}
\newtheorem{lemma}[theorem]{Lemma}
\newtheorem{corollary}[theorem]{Corollary}
\theoremstyle{definition}

\newtheorem{remark}[theorem]{Remark}

\usepackage{etoolbox}
\newtheorem{formulation}{}

\AtEndEnvironment{formulation}{\hfill\(\Diamond\)\par}

\newtheorem{formulationr}{}

\AtEndEnvironment{formulationr}{\hfill\(\Diamond\)\par}

\newcommand{\R}{\mathbb{R}}
\newcommand{\Xh}{X_\Delta}
\newcommand{\Dt}{\Delta t}
\newcommand{\Dx}{\Delta x}
\newcommand{\Ltt}[1]{L_{tt}^{(#1)}}
\newcommand{\Lxx}{L_{xx}}

\title[Riesz-Preconditioned Extragradient for Price-Formation MFGs]{A Riesz-Map Preconditioned Extragradient Method for Price Formation in Mean Field Games}
\author{Yeva Gevorgyan}
\address{King Abdullah University of Science and Technology (KAUST), Saudi Arabia}
\email{yeva.gevorgyan@kaust.edu.sa}

\author{Diogo Gomes}
\address{King Abdullah University of Science and Technology (KAUST), Saudi Arabia}
\email{diogo.gomes@kaust.edu.sa}

\date{}

\begin{document}
	
\begin{abstract}
	We develop and analyze a projected extragradient method, preconditioned by a Riesz map, for a first-order price-formation mean field game in one space dimension. The model couples
	Hamilton--Jacobi--Bellman and transport equations through a global
	market-clearing constraint. Finite propagation justifies truncating
	the first-order transport equation to a bounded interval.
	A monotone weak formulation incorporates the mixed initial--terminal
	data. Its Riesz map uses $L^2$ norms for density and price
	directions and a space-time $H^1$ norm for value-function
	directions, preconditioning the residual without adding higher-order
	regularization terms to the MFG equations. A density-only argument
	rules out cocoercivity on the localized admissible set in every
	equivalent Hilbert metric, motivating extragradient updates.
	For $H(x,p)=p^2/2$, we use a monotone upwind discretization with
	adjoint-consistent transport and a state-constraint boundary closure.
	We establish a mesh-independent Lipschitz bound on bounded convex
	admissible sets with uniform density, slope, and price bounds.
	With exact metric projection onto these sets, the iteration converges
	to a solution of the discrete variational inequality on each fixed
	grid, with an admissible step-size threshold independent of both
	mesh sizes. Convergence under mesh refinement is not proved.
	On a quadratic benchmark, all reported runs reach the stopping
	tolerance with one fixed step size across a factor-of-16 reduction
	in both mesh sizes. Price errors and value-function errors measured
	away from low-density regions decay approximately at first order
	under refinement, while density errors decrease more slowly.
\end{abstract}
	
	
	\maketitle
	
	\section{Introduction}
	\label{sec:intro}
	
	Mean field games (MFGs), introduced independently by Lasry and
	Lions~\cite{ll1,ll2,lasryMeanFieldGames2007} and Huang, Caines, and
	Malham\'e~\cite{huangLargePopulationStochastic2006,Caines2}, describe
	strategic interactions in large populations.
	Applications and related equilibrium models include crowd
	motion~\cite{achdouMeanFieldGames2019}, intraday electricity
	trading~\cite{aidEquilibriumPriceIntraday2022},
	renewable-capacity investment~\cite{alasseurMeanFieldModel2023},
	and speculative oil storage~\cite{achdouClassShorttermModels2022}.
We study the price-formation MFG introduced by Gomes and
Sa\'ude~\cite{gomesMeanFieldGameApproach2021}, in which agents
trade a commodity at a price determined by aggregate supply
and demand. Our aim is to develop and analyze a monotonicity-based
numerical method for its first-order formulation.
	
    Stochastic-supply
	extensions were developed in~\cite{GoGuRi2021,GoGuRi2021B},
	a measure-theoretic duality formulation in~\cite{YTDJduality2021},
	and machine-learning solvers with common noise
	in~\cite{gomes2023machineB}.
	Related price models include the parabolic free-boundary model
	of~\cite{MR2573148} and the game-theoretic order-book model
	of~\cite{evangelistaSaporitoThamsten2022}.
	We first state the first-order model formally. The standing assumptions
	are given in Section~\ref{sec:formulation}, and the weak and
	Hilbert-space realizations in Section~\ref{sec:monotone-ops}.	
	\begin{formulation}\label{form:price-mfg}
		Let \(T>0\). Given a continuous Hamiltonian
		\(H:\R\times\R\to\R\) that is continuously differentiable in the
		momentum variable \(p\), a supply function
		\(Q\in C([0,T])\), a terminal cost
		\(\bar u:\R\to\R\), and an initial probability density
		\(\bar m\in L^1(\R)\), with
		\[
		\bar m\ge 0,
		\qquad
		\int_\R \bar m(x)\,dx=1,
		\]
		we would like to find a value function \(u=u(x,t)\), a population density
		\(m=m(x,t)\), and a scalar price \(\varpi=\varpi(t)\) satisfying
		\begin{equation}\label{eq:det-system}
			\begin{cases}
				-u_t + H\bigl(x,\varpi(t)+u_x\bigr)=0,
				& (x,t)\in\R\times(0,T), \\[2mm]
				m_t-\Bigl(mD_pH\bigl(x,\varpi(t)+u_x\bigr)\Bigr)_x=0,
				& (x,t)\in\R\times(0,T), \\[2mm]
				\displaystyle
				\int_\R
				m(x,t)D_pH\bigl(x,\varpi(t)+u_x(x,t)\bigr)\,dx
				+Q(t)=0,
				& t\in(0,T),
			\end{cases}
		\end{equation}
		together with
		\begin{equation}\label{eq:det-bc}
			m(x,0)=\bar m(x),
			\qquad
			u(x,T)=\bar u(x).
		\end{equation}
		We look for solutions for which \(m(\cdot,t)\) remains a
		probability density and the market-clearing integral is finite.
		The state \(x\) represents a signed inventory or capital position.
		The backward Hamilton--Jacobi--Bellman (HJB) equation determines
		the optimal trading rate
		\(\alpha^*=-D_pH(x,\varpi+u_x)\), and the forward transport
		equation describes the resulting population distribution.
		Market clearing requires the population-average trading rate
		to equal the prescribed supply \(Q(t)\).
\end{formulation}

 Under
	\ref{H:convex}, 
    the left-hand side of the third equation in
\eqref{eq:det-system} is strictly increasing in \(\varpi\), so the price is unique for fixed
	\(m(\cdot,t)\) and \(u_x(\cdot,t)\).
	We develop the numerical method for the first-order
	system~\eqref{eq:det-system}; the viscous regularization in
	Section~\ref{sec:formulation} supplies the analytical estimates used for
    domain truncation and localization.
    
 Our focus is a monotonicity-preserving discretization of the coupled
 PDE residual, with a Hilbert metric that controls the transport
 derivatives without inverse mesh factors. We seek a computable
 Riesz representative with a mesh-independent Lipschitz bound on
 sets with uniform density, slope, and price bounds. The corresponding
 extragradient convergence argument requires projection in the same metric.   
 
   Related monotonicity-based constructions include
	\cite{almullaTwoNumericalApproaches2017}, which constructs an
	\(L^2\)-contractive flow with an adjoint transport discretization in a
	stationary continuous-state setting, and~\cite{GJS2}, which treats
	finite-state time-dependent MFGs using an \(H^1\) representation of
	functionals. 
	Here we analyze the coupled residual in \((m,u,\varpi)\),
	including the market-clearing block and the state-constraint
	boundary rows, in a mixed \(L^2\)--\(H^1\) metric.
	The mesh-independent estimate in Proposition~\ref{prop:lip}
	yields the extragradient convergence result in
	Corollary~\ref{cor:mesh-independent-step}.
	Other MFG discretizations include finite-difference schemes and their
	convergence analysis
	\cite{achdouMeanFieldGames2010,achdou2013finite,Achdou2015}
	and monotone-inclusion methods for second-order non-potential
	MFGs, including metric preconditioning for grid-independent
	algorithmic step sizes~\cite{nurbekyanMonotoneInclusionMethods2024}.
	For price-formation models, \cite{wangFungNurbekyan2025}
	develops a primal-dual method based on a variational
	characterization of equilibrium prices, and
	\cite{ashrafyanFullyDiscreteSemiLagrangianScheme2025} 
	uses monotonicity to establish convergence of a fully discrete
	semi-Lagrangian price scheme to a weak solution under mesh refinement.
    The
	extragradient iteration goes back to~\cite{korpelevich1976}; see
	\cite{facchinei2003finite,bauschke2017convex} for the general
	variational-inequality framework.

	We formulate the truncated first-order system as a monotone weak
	operator. Riesz representation preserves monotonicity in every
	equivalent Hilbert metric, but a density-only argument rules out
	cocoercivity on the localized admissible set in every such metric.
	This obstruction motivates the projected extragradient method.
	Our Riesz map represents the HJB and market-clearing residuals in
	the \(L^2\) density and price components, respectively, and lifts
	the transport residual through a space-time \(H^1\) value-function
	norm.
	
	For \(H(x,p)=p^2/2\), we use a monotone upwind discretization with
	adjoint transport. In the mixed \(L^2\)--Sobolev metric, its Riesz
	representative is uniformly Lipschitz on admissible sets with
	mesh-independent density, slope, and price bounds. We prove that
	the extragradient method with exact metric projection onto the
	bounded, slope-constrained admissible sets converges to a solution
	of the corresponding discrete variational inequality (VI) on each fixed
	grid, with an admissible step-size threshold independent of
	\(\Delta x\) and \(\Delta t\). At a discrete VI solution with
	inactive upper density bound, the HJB equation holds where the
	density is positive, while only the subsolution inequality is
	imposed in vacuum.
	
	In the reported experiments, all initial, predictor, and corrector
	states satisfy the uniform bounds, and value-function clipping is
	inactive. Thus each computed trajectory is a finite initial segment
	of the exact slope-constrained metric-projection iteration.
	On the quadratic benchmark, the price and value-function errors
	exhibit approximately first-order decay under refinement, with
	the value-function error measured away from low-density regions;
	the density error decreases more slowly. Convergence under mesh
	refinement is not established here.
	
The paper is organized as follows.
Sections~\ref{sec:formulation}--\ref{sec:monotone-ops} develop the
continuous analytical and operator framework.
Sections~\ref{sec:discrete}--\ref{sec:eg} construct and analyze the
fully discrete method.
Sections~\ref{sec:benchmark} and~\ref{sec:numerics} present the
quadratic benchmark and numerical experiments.
	
	\section{Analytical Preliminaries and Domain Truncation}
	\label{sec:formulation}

We introduce the viscous regularization and the standing
assumptions, establish the estimates used later, and justify the
bounded-domain reduction by finite propagation.
	
	\subsection{Viscous regularization}
	\label{sec:problem}
	
	We consider the following regularized model. 
	
	\begin{formulationr}\label{form:price-mfgr}
		Consider the setting of~\ref{form:price-mfg}, and let
		\(\epsilon>0\). Find
		\((u^\epsilon,m^\epsilon,\varpi^\epsilon)\) satisfying
		\begin{equation}\label{eq:reg-system}
			\begin{cases}
				-u^\epsilon_t
				+H\bigl(x,\varpi^\epsilon(t)+u^\epsilon_x\bigr)
				=\epsilon u^\epsilon_{xx},
				& (x,t)\in\R\times(0,T), \\[2mm]
				m^\epsilon_t
				-\Bigl(
				m^\epsilon
				D_pH\bigl(x,\varpi^\epsilon(t)+u^\epsilon_x\bigr)
				\Bigr)_x
				=\epsilon m^\epsilon_{xx},
				& (x,t)\in\R\times(0,T), \\[2mm]
				\displaystyle
				\int_\R
				m^\epsilon(x,t)
				D_pH\bigl(
				x,\varpi^\epsilon(t)+u^\epsilon_x(x,t)
				\bigr)\,dx
				+Q(t)=0,
				& t\in(0,T),
			\end{cases}
		\end{equation}
		with
		\begin{equation}\label{eq:reg-bc}
			m^\epsilon(x,0)=\bar m(x),
			\qquad
			u^\epsilon(x,T)=\bar u(x).
		\end{equation}
	\end{formulationr}

	Under sufficient integrability and decay at infinity,
	\(m^\epsilon(\cdot,t)\) remains a probability density. Since positive
	diffusion generally destroys compact support, the viscous problem
	remains posed on \(\R\); the bounded-domain reduction is made only
	after the vanishing-viscosity limit.
	
	\subsection{Main assumptions}
	\label{sec:assumptions}
	
	We use the following hypotheses, adapted
	from~\cite{gomesMeanFieldGameApproach2021}.
	
	\begin{enumerate}[label=\textup{(H\arabic*)},leftmargin=3em]
		\item\label{H:structure}
		The Hamiltonian is of the form
		$H(x,p)=H_0(p)-V(x)$, where $H_0:\R\to\R$ is $C^3$ and
		$V:\R\to\R$ is $C^2$ and bounded from below.
		In particular, $D_pH=H_0'$ is independent of $x$.
		
		\item\label{H:convex}
		There exist constants \(0<\theta\le\Theta\) such that
		\[
		\theta\le H_0''(p)\le\Theta
		\qquad\text{for all }p\in\R.
		\]
		Thus \(p\mapsto H(x,p)\) is uniformly convex, and \(H_0'\) is
		globally Lipschitz with constant \(\Theta\).
		
		\item\label{H:third-deriv}
		There exists \(C_3\ge0\) such that
		\(|H_0'''(p)|\le C_3\) for all \(p\in\R\).
		
		\item\label{H:potential}
		The terminal cost \(\bar u\in C^2(\R)\), and \(V,\bar u\) are
		globally Lipschitz on~\(\R\), with
		$\|V''\|_{L^\infty(\R)}\le C_V$ and
		$\|\bar u''\|_{L^\infty(\R)}\le C_{\bar u}$ for some
		$C_V,C_{\bar u}\ge 0$.
		
		\item\label{H:data}
		The initial density \(\bar m\in C_c^2(\R)\) satisfies
		\(\bar m\ge0\) and \(\int_\R\bar m\,dx=1\), and
		\(Q\in C^1([0,T])\).
	\end{enumerate}
	
	The assumptions on \(V\) and \(\bar u\) yield the spatial Lipschitz
	and semiconcavity estimates used below; see
	\cite{gomesMeanFieldGameApproach2021}.	
	The benchmark uses a quadratic terminal cost, which is not globally
	Lipschitz on \(\R\) and is therefore treated separately from the
	whole-line theory under~\ref{H:potential}; see
	Section~\ref{sec:benchmark}.

	\subsection{Price regularity}
	\label{sec:price-regularity}
	
	We show that every vanishing-viscosity limit \(\varpi\) is Lipschitz
	continuous under \ref{H:structure}--\ref{H:data}.
	The proof uses the uniform estimates
	\begin{equation}\label{eq:unif-lip-r}
		M_L
		:=
		\sup_{0<\epsilon\le1}
		\sup_{t\in[0,T]}
		\|u_x^\epsilon(\cdot,t)\|_{L^\infty(\R)}		
		<\infty
	\end{equation}
	and
	\begin{equation}\label{eq:unif-energy-r}
		E
		:=
		\sup_{0<\epsilon\le1}
		\int_0^T\!\!\int_\R
		H_0''(\varpi^\epsilon+u_x^\epsilon)\,
		(u_{xx}^\epsilon)^2\,m^\epsilon\,dx\,dt
		<\infty, 
	\end{equation}
	which follow from
	\cite[Propositions~1 and~5]{gomesMeanFieldGameApproach2021},
	respectively.

	\begin{proposition}[Uniform price bounds and Lipschitz limits]
		\label{prop:price-lipschitz}
		Consider the \ref{form:price-mfgr}.
		Assume \ref{H:structure}--\ref{H:data}. For
		\(0<\epsilon\le1\), let \((u^\epsilon,m^\epsilon,\varpi^\epsilon)\)
		be smooth solutions of the viscous
		system~\eqref{eq:reg-system}--\eqref{eq:reg-bc}, satisfying
		\eqref{eq:unif-lip-r} and \eqref{eq:unif-energy-r}, such that
		 \(m^\epsilon(\cdot,t)\) is a probability density for every
		\(t\in[0,T]\), with sufficient decay at spatial infinity to justify
		the integrations by parts below.
		Then the following hold.
		\begin{enumerate}[label=\textup{(\alph*)},leftmargin=3em]
			\item\label{item:price-w11-r}
			The family \((\varpi^\epsilon)_{0<\epsilon\le1}\) is uniformly
			bounded in \(W^{1,1}(0,T)\).
			\item\label{item:price-compact-r}
			Every sequence \(\epsilon_n\downarrow0\) admits a subsequence,
			not relabeled, along which \(\varpi^{\epsilon_n}\to\varpi\)
			uniformly on \([0,T]\).
			\item\label{item:price-limit-r}
			Every such limit satisfies \(\varpi\in W^{1,\infty}(0,T)\), with
			\[
			\|\dot\varpi\|_{L^\infty(0,T)}
			\le
			\frac{\|\dot Q\|_{L^\infty(0,T)}}{\theta}
			+\|V'\|_{L^\infty(\R)} .
			\]
		\end{enumerate}
	\end{proposition}
	
	\begin{proof}
		Let
		\[
		p^\epsilon(x,t)
		=
		\varpi^\epsilon(t)+u^\epsilon_x(x,t).
		\]
		Differentiating the market-clearing condition in
		\eqref{eq:reg-system}, using the HJB and transport equations to
		eliminate \(u_{xt}^\epsilon\) and \(m_t^\epsilon\), integrating by
		parts in \(x\), and solving for \(\dot\varpi^\epsilon\), we obtain
		\begin{equation}\label{eq:price-ode-r}
			\dot\varpi^\epsilon
			= \frac{\displaystyle
				-\dot Q
				+ \int_\R H_0''(p^\epsilon)\,V'(x)\,
				m^\epsilon\,dx
				- \epsilon\!\int_\R
				H_0'''(p^\epsilon)\,
				(u_{xx}^\epsilon)^2\,m^\epsilon\,dx}
			{\displaystyle
				\int_\R H_0''(p^\epsilon)\,
				m^\epsilon\,dx}\,.
		\end{equation}
		
		Since \(H_0''\ge\theta\) by~\ref{H:convex} and
		\(m^\epsilon(\cdot,t)\) is a probability density, the denominator
		in \eqref{eq:price-ode-r} is bounded below by~\(\theta\); hence the
		\(\dot Q\)-term is bounded by
		\(\|\dot Q\|_{L^\infty(0,T)}/\theta\). 
		The potential term is bounded by
		\(\|V'\|_{L^\infty(\R)}\), since it is a weighted average of \(V'\).
		For the viscous term,
		\(|H_0'''|\le C_3\) by~\ref{H:third-deriv} and \(H_0''\ge\theta\)
		give
		\[
		\left\|
		\frac{
			\epsilon\int_\R
			H_0'''(p^\epsilon)\,
			(u_{xx}^\epsilon)^2\,m^\epsilon\,dx
		}{
			\int_\R
			H_0''(p^\epsilon)\,m^\epsilon\,dx
		}
		\right\|_{L^1(0,T)}
		\le
		\frac{\epsilon C_3}{\theta^2}\,E
		=O(\epsilon),
		\]
		by the energy estimate~\eqref{eq:unif-energy-r}. Writing the
		first two terms in \eqref{eq:price-ode-r} as
		\(b^\epsilon\) and the viscous term as \(r^\epsilon\), we obtain
		\begin{equation}\label{eq:price-split-r}
			\dot\varpi^\epsilon=b^\epsilon+r^\epsilon,
			\qquad
			\|b^\epsilon\|_{L^\infty(0,T)}
			\le
			\frac{\|\dot Q\|_{L^\infty(0,T)}}{\theta}
			+\|V'\|_{L^\infty(\R)},
			\qquad
			\|r^\epsilon\|_{L^1(0,T)}\to 0 .
		\end{equation}
		
		By market clearing, \eqref{eq:unif-lip-r}, and the
		\(\Theta\)-Lipschitz continuity of \(H_0'\),
		\[
		|H_0'(\varpi^\epsilon(t))+Q(t)|\le\Theta M_L.
		\]
		Since \(H_0''\ge\theta\), it follows that
		\[
		\|\varpi^\epsilon\|_{L^\infty(0,T)}
		\le
		\frac{\Theta M_L+\|Q+H_0'(0)\|_{L^\infty(0,T)}}{\theta}.
		\]
		Together with \eqref{eq:price-split-r}, this proves the uniform
		\(W^{1,1}(0,T)\) bound in~\ref{item:price-w11-r}.

		Set
		\(B^\epsilon(t)=\varpi^\epsilon(0)+\int_0^t b^\epsilon(s)\,ds\).
		By the preceding bounds, the family \(B^\epsilon\) is uniformly bounded
		and equi-Lipschitz. Hence, along any sequence
		\(\epsilon_n\downarrow0\), a subsequence converges uniformly on
		\([0,T]\) by the Arzel\`a--Ascoli theorem. Since
		\[
		\sup_{t\in[0,T]}
		\left|
		\varpi^\epsilon(t)-B^\epsilon(t)
		\right|
		=
		\sup_{t\in[0,T]}
		\left|
		\int_0^t r^\epsilon(s)\,ds
		\right|
		\le
		\|r^\epsilon\|_{L^1(0,T)}
		\longrightarrow0,
		\]
		the same subsequence of \(\varpi^{\epsilon_n}\) converges uniformly
		to a limit \(\varpi\), which proves~\ref{item:price-compact-r}.
			Since the functions \(B^{\epsilon_n}\) have the common Lipschitz
			bound in \eqref{eq:price-split-r}, their uniform limit
			\(\varpi\) has the same bound. This proves
			\ref{item:price-limit-r}.
	\end{proof}
	
	\subsection{Semiconcavity and a density bound}
	\label{sec:semboundm}

	Semiconcavity yields an \(L^\infty\) bound for the viscous densities,
	uniformly in~\(\epsilon\); both estimates pass to the corresponding
	vanishing-viscosity limits.
	The uniform bounds on \(u_x^\epsilon\) and \(\varpi^\epsilon\) give
	\begin{equation}\label{eq:momentum-bound}
		\sup_{0<\epsilon\le1}
		\|\varpi^\epsilon+u_x^\epsilon\|_{L^\infty(\R\times(0,T))}
		\le M_p
	\end{equation}
	for some \(M_p>0\).
	
	\begin{proposition}[Semiconcavity and an \(L^\infty\) density bound]
		\label{prop:semiconcavity-density}
		Assume \ref{H:structure}--\ref{H:data}, and let
		\((u^\epsilon,m^\epsilon,\varpi^\epsilon)\), \(0<\epsilon\le1\), be a
		family of smooth solutions as in
		Proposition~\ref{prop:price-lipschitz}, satisfying the uniform momentum
		bound~\eqref{eq:momentum-bound}. There is a constant
		\(C_{\mathrm{sc}}\ge0\), depending only on the data and \(T\),
		such that \(u^\epsilon\) satisfies
		\begin{equation}\label{eq:semiconcavity-estimate}
			u^\epsilon(x+h,t)-2u^\epsilon(x,t)+u^\epsilon(x-h,t)
			\le
			C_{\mathrm{sc}}h^2
		\end{equation}
		for all \(x,h\in\R\), \(t\in[0,T]\), and \(0<\epsilon\le1\).
		Moreover, for every \(t\in[0,T]\),
		\begin{equation}\label{eq:density-linfty-uniform}
			\|m^\epsilon(\cdot,t)\|_{L^\infty(\R)}
			\le
			\|\bar m\|_{L^\infty(\R)}e^{\Theta C_{\mathrm{sc}}t}
			\le
			\|\bar m\|_{L^\infty(\R)}e^{\Theta C_{\mathrm{sc}}T},
		\end{equation}
		uniformly in \(0<\epsilon\le1\).
		Consequently, every locally uniform limit of \(u^\epsilon\) is
		semiconcave in \(x\), and every weak-star limit of \(m^\epsilon\) in
		\(L^\infty(\R\times(0,T))\) satisfies the final uniform bound in
		\eqref{eq:density-linfty-uniform}.
	\end{proposition}
	
	\begin{proof}
		The argument of
		\cite[Proposition~2]{gomesMeanFieldGameApproach2021} gives
		\eqref{eq:semiconcavity-estimate}. Since \(u^\epsilon\) is smooth,
		this implies
		\begin{equation}\label{eq:uxx-semiconcavity}
			u_{xx}^\epsilon(x,t)\le C_{\mathrm{sc}}.
		\end{equation}
		
		Using \(D_pH(x,p)=H_0'(p)\), the transport-diffusion equation can be
		written as
		\begin{equation}\label{eq:expanded-fp}
			m^\epsilon_t
			-
			H_0'(\varpi^\epsilon+u_x^\epsilon)m^\epsilon_x
			-
			\epsilon m^\epsilon_{xx}
			=
			H_0''(\varpi^\epsilon+u_x^\epsilon)
			u_{xx}^\epsilon m^\epsilon .
		\end{equation}
		By \ref{H:convex} and \eqref{eq:uxx-semiconcavity},
		\[
		H_0''(\varpi^\epsilon+u_x^\epsilon)
		u_{xx}^\epsilon
		\le
		\Theta C_{\mathrm{sc}}.
		\]
		By \eqref{eq:momentum-bound}, the drift in
		\eqref{eq:expanded-fp} is uniformly bounded. Applying the parabolic
		comparison principle on \(\R\), using the assumed spatial decay, to
		\[
		e^{-\Theta C_{\mathrm{sc}}t}m^\epsilon
		\]
		gives the first inequality in
		\eqref{eq:density-linfty-uniform}; the second follows from
		\(t\le T\).
		
		Finally, local uniform convergence preserves
		\eqref{eq:semiconcavity-estimate}, and the final uniform bound in
		\eqref{eq:density-linfty-uniform} passes to every weak-star limit in
		\(L^\infty(\R\times(0,T))\).
	\end{proof}

	\subsection{Domain truncation}
	\label{sec:truncation-price}
	
	For the regularized solutions considered above, set
	\[
	v^\epsilon(x,t)
	=
	-D_pH\bigl(x,\varpi^\epsilon(t)+u_x^\epsilon(x,t)\bigr).
	\]
	By \ref{H:structure} and \eqref{eq:momentum-bound}, there is
	\(B\ge0\), independent of \(\epsilon\), such that
	\begin{equation}\label{eq:bounded-feedback}
		\sup_{0<\epsilon\le1}
		\|v^\epsilon\|_{L^\infty(\R\times(0,T))}
		\le B.
	\end{equation}
	The density equation is
	\[
	m^\epsilon_t+(v^\epsilon m^\epsilon)_x
	=
	\epsilon m^\epsilon_{xx}.
	\]
	
	Choose \(R_0>0\) such that
	\(\operatorname{supp}\bar m\subset[-R_0,R_0]\).
	Let \(\chi:\R\to[0,1]\) be smooth and nondecreasing, with
	\(\chi=0\) on \((-\infty,0]\), \(\chi>0\) on \((0,\infty)\),
	and \(\chi=1\) on \([1,\infty)\). Set
	\[
	\varphi(x,t)=\chi(|x|-R_0-Bt).
	\]
	This function is smooth, since it vanishes near \(x=0\).
	Testing the density equation with \(\varphi\), using spatial
	cutoffs at infinity, gives
	\[
	\begin{aligned}
		\frac{d}{dt}\int_\R \varphi m^\epsilon\,dx
		&=
		\int_\R
		\bigl(\varphi_t+v^\epsilon\varphi_x
		+\epsilon\varphi_{xx}\bigr)m^\epsilon\,dx\\
		&\le
		\epsilon\|\chi''\|_{L^\infty(\R)}.
	\end{aligned}
	\]
	Indeed,
	\(\varphi_t+v^\epsilon\varphi_x
	=\chi'(|x|-R_0-Bt)(v^\epsilon\operatorname{sgn}x-B)\le0\),
	and \(m^\epsilon\) has unit mass. Since
	\(\varphi(\cdot,0)\bar m=0\), integration in time yields
	\[
	0\le
	\int_\R\varphi(x,t)m^\epsilon(x,t)\,dx
	\le
	\epsilon t\|\chi''\|_{L^\infty(\R)}.
	\]
	Integrating this estimate in time and passing to a weak-star
	limit \(m\) in \(L^\infty(\R\times(0,T))\) as \(\epsilon\to0\),
	after localization in space, gives \(\varphi m=0\) almost everywhere.
	For a first-order limit with a narrowly continuous
	probability-measure representative, the map
	\(t\mapsto\int_\R\varphi(x,t)m(x,t)\,dx\) is continuous
	and hence vanishes for every \(t\in[0,T]\). Therefore,
\[
\operatorname{supp}m(\cdot,t)
\subset[-R_0-Bt,R_0+Bt],
\qquad t\in[0,T].
\]
The same calculation with \(\epsilon=0\) proves this inclusion
directly for any narrowly continuous nonnegative weak solution
with initial density \(\bar m\) and velocity
\(v=-D_pH(x,\varpi+u_x)\) satisfying
\(|v|\le B\) for \(m\,dx\,dt\)-almost every \((x,t)\).	
	
	Choose a bounded interval \(\Omega=(a,b)\) such that
	\[
	[-R_0-BT,R_0+BT]\Subset\Omega.
	\]
	Then both \(m\) and
	\(mD_pH(x,\varpi+u_x)\) vanish in a neighborhood of
	\(\partial\Omega\times[0,T]\). Thus the restriction of \(m\) to
	\(\Omega\times[0,T]\) satisfies the transport equation with zero
	boundary flux.
	The whole-line value function may likewise be restricted to
	\(\Omega\); a standalone bounded-domain HJB formulation requires a compatible boundary convention, specified for the
	discrete scheme in Section~\ref{sec:ham}.

	\section{Monotone Operators in Hilbert Spaces}
	\label{sec:hilbert-monotone}

	Let \(\mathcal X\) be a real Hilbert space with inner product
	\((\cdot,\cdot)_{\mathcal X}\), induced norm
	\(\|\cdot\|_{\mathcal X}\), and dual \(\mathcal X'\). We denote the
	duality pairing by \(\langle\cdot,\cdot\rangle\).
	
	The Riesz operator
	\(\mathcal R:\mathcal X'\to\mathcal X\) is characterized by
	\begin{equation}\label{eq:abstract-riesz}
		(\mathcal R f,v)_{\mathcal X}
		=
		\langle f,v\rangle
		\qquad
		\forall\,v\in\mathcal X , 
	\end{equation}
	for \(f\in\mathcal X'\).
	Moreover, \(\mathcal R\) is an isometric isomorphism:
	\[
	\|\mathcal R f\|_{\mathcal X}=\|f\|_{\mathcal X'},
	\]
	where \(\|\cdot\|_{\mathcal X'}\) is the dual norm induced by
	\(\|\cdot\|_{\mathcal X}\).
	
	Let \(\mathcal A=\bar\omega+\mathcal X\) be an affine space,
	let \(\mathcal K\subset\mathcal A\) be convex, and let
	\(A:\mathcal K\to\mathcal X'\). The operator \(A\) is monotone if
	\[
	\langle A(\omega)-A(\widehat\omega),\,
	\omega-\widehat\omega\rangle
	\ge 0
	\qquad
	\forall\,\omega,\widehat\omega\in\mathcal K .
	\]
	
	For \(A:\mathcal K\to\mathcal X'\), its Riesz representative
	\(\mathcal R A:\mathcal K\to\mathcal X\) satisfies
	\begin{equation}\label{eq:abstract-lift-identity}
		(\mathcal R A(\omega)-\mathcal R A(\widehat\omega),\,
		\omega-\widehat\omega)_{\mathcal X}
		=
		\langle A(\omega)-A(\widehat\omega),\,
		\omega-\widehat\omega\rangle .
	\end{equation}
	Hence \(A\) is monotone in the duality pairing if and only if
	\(\mathcal R A\) is monotone in
	\((\cdot,\cdot)_{\mathcal X}\).
	
	Likewise, for \(\omega^*\in\mathcal K\),
	\eqref{eq:abstract-riesz} gives
	\[
	\langle A(\omega^*),\eta-\omega^*\rangle
	=
	(\mathcal R A(\omega^*),\eta-\omega^*)_{\mathcal X}
	\qquad
	\forall\,\eta\in\mathcal K.
	\]
	Therefore, the dual and Riesz formulations of the variational
	inequality are equivalent.
	
		In finite dimensions, every inner product can be written as
		\[
		(v,w)_{\mathcal X}=\langle Bv,w\rangle,
		\qquad
		\mathcal R=B^{-1},
		\]
		where \(B:\mathcal X\to\mathcal X'\) is represented by a symmetric positive definite matrix.

	The following saddle-gradient criterion motivates the saddle-point
	interpretation of Remark~\ref{rem:saddle-point} and is applied
	to the discrete operator in Proposition~\ref{prop:disc-mono}.
	
	\begin{lemma}[Monotonicity of saddle gradients]
		\label{lem:saddle-gradient-monotone}
		Suppose that
		\(\mathcal X=\mathcal X_-\times\mathcal X_+\), with corresponding
		variables \(\omega=(x,y)\), 
		and let \(\mathcal J\colon\mathcal O\to\R\) be Fréchet differentiable on an open set \(\mathcal O\) containing a nonempty convex product set
		\(\mathcal K=\mathcal K_-\times\mathcal K_+\).
		Assume that \(\mathcal J(\cdot,y)\) is concave on
		\(\mathcal K_-\) for every \(y\in\mathcal K_+\), and that
		\(\mathcal J(x,\cdot)\) is convex on \(\mathcal K_+\) for every
		\(x\in\mathcal K_-\).
		Then the
		saddle-gradient
		\[
		A(x,y)
		=
		\bigl(-D_x\mathcal J(x,y),\,D_y\mathcal J(x,y)\bigr)
		\]
		defines a monotone map \(A:\mathcal K\to\mathcal X'\).
		On an affine product space
		\((\bar x+\mathcal X_-)\times(\bar y+\mathcal X_+)\), the same
		conclusion follows by applying the result to
		\((\xi,\eta)\mapsto\mathcal J(\bar x+\xi,\bar y+\eta)\)
		on the translated admissible set.
	\end{lemma}
	\begin{proof}
		For
		\(\omega=(x,y),\widehat\omega=(\widehat x,\widehat y)\in\mathcal K\),
		set
		\(\delta x=x-\widehat x\), \(\delta y=y-\widehat y\), and
		\[
		\Delta\mathcal J
		=
		\mathcal J(x,y)-\mathcal J(x,\widehat y)
		-\mathcal J(\widehat x,y)+\mathcal J(\widehat x,\widehat y).
		\]
		The supporting-hyperplane inequalities for convexity in \(y\) give
		\[
		\bigl\langle
		D_y\mathcal J(x,y)-D_y\mathcal J(\widehat x,\widehat y),
		\delta y
		\bigr\rangle
		\ge \Delta\mathcal J,
		\]
		whereas concavity in \(x\) gives
		\[
		-\bigl\langle
		D_x\mathcal J(x,y)-D_x\mathcal J(\widehat x,\widehat y),
		\delta x
		\bigr\rangle
		\ge -\Delta\mathcal J.
		\]
		Adding the two inequalities gives
		\[
		\langle A(\omega)-A(\widehat\omega),
		\omega-\widehat\omega\rangle\ge0,
		\]
		which proves the claim.
	\end{proof}
	
	\section{Continuous Monotone Operator and Riesz Representative}
	\label{sec:monotone-ops}
	
	We formulate the truncated first-order price-formation MFG
	as a monotone weak operator, introduce its Riesz representative in a
	concrete Hilbert realization, and identify a density-only obstruction
	to cocoercivity that motivates the discrete extragradient method.
	
	\subsection{Monotone operator formulation}
	\label{sec:monotone-cont}
	
	Let \(\Omega_T=\Omega\times(0,T)\), and let \(\mathcal M\) be a
	Banach space of density directions on~\(\Omega_T\). Let \(\mathcal U\) be a Banach space of value-function directions with
	weak derivatives \(\tilde u_t,\tilde u_x\), well-defined traces, and
	\(\tilde u(\cdot,T)=0\). Let \(\mathcal P\) be a Banach space of price
	directions, and set
	\[
	\mathcal X
	=
	\mathcal M\times\mathcal U\times\mathcal P .
	\]
	For now, we leave these spaces unspecified; the weak
	action~\eqref{eq:cont-A-weak} and the monotonicity of
	Proposition~\ref{prop:cont-mono} involve only the duality pairing
	between \(\mathcal X\) and \(\mathcal X'\), so no Hilbert inner product is needed at this stage.
	Fix a space-time extension of the terminal cost, still denoted by
	\(\bar u\), and define
	\[
	\mathcal A
	=
	\mathcal M\times(\bar u+\mathcal U)\times\mathcal P .
	\]
	We consider states
	\(\omega=(m,u,\varpi)\in\mathcal A\) with \(m\ge0\) almost
	everywhere. 
	The terminal trace condition is encoded in the definition of
	\(\mathcal A\), whereas \(m(\cdot,0)=\bar m\) is encoded by the
	boundary term in the weak transport formulation, so no initial trace
	of \(m\) is required.
	
	For \(\omega=(m,u,\varpi)\in\mathcal A\) and
	\(\tilde\omega=(\tilde m,\tilde u,\tilde\varpi)\in\mathcal X\),
	define the weak action of \(A(\omega)\) by
	\begin{align}
		\langle A(\omega),\tilde\omega\rangle
		&=
		\int_0^T\!\!\int_\Omega
		\bigl(u_t-H(x,\varpi+u_x)\bigr)\tilde m\,dx\,dt
		\notag\\
		&\quad
		+ \int_0^T\!\!\int_\Omega
		m\bigl(-\tilde u_t+D_pH(x,\varpi+u_x)\tilde u_x\bigr)\,dx\,dt
		-\int_\Omega \bar m(x)\tilde u(x,0)\,dx
		\notag\\
		&\quad
		+ \int_0^T \tilde\varpi(t)
		\Bigl(\int_\Omega m\,D_pH(x,\varpi+u_x)\,dx+Q(t)\Bigr)\,dt .
		\label{eq:cont-A-weak}
	\end{align}
	Whenever the right-hand side is bounded with respect to the norm of
	\(\mathcal X\), it defines an element \(A(\omega)\in\mathcal X'\).
	In the transport component, the derivatives fall on \(\tilde u\); the
	term at \(t=0\) encodes the initial density, and
	\(\tilde u(\cdot,T)=0\) removes the terminal term. The absence of a
	spatial boundary term corresponds to the zero-flux truncation of
	Section~\ref{sec:truncation-price}.
	
	For sufficiently regular states and rich enough test spaces,
	\(A(\omega)=0\) is the weak formulation of the HJB, transport, and
	market-clearing equations, with the initial and terminal data encoded
	as above. No spatial boundary condition for \(u\) is imposed here.

	\begin{proposition}[Monotonicity of the weak operator]
		\label{prop:cont-mono}
		Assume \ref{H:structure} and \ref{H:convex}. Let
		\(\omega,\widehat\omega\in\mathcal A\) have nonnegative
		density components, and suppose that \eqref{eq:cont-A-weak} defines
		\(A(\omega),A(\widehat\omega)\in\mathcal X'\). Then
		\begin{equation}\label{eq:cont-mono}
			\langle A(\omega)-A(\widehat\omega),
			\omega-\widehat\omega\rangle
			\ge0.
		\end{equation}
	\end{proposition}
	
	\begin{proof}
		Write
		\[
		z=\varpi+u_x,
		\qquad
		\widehat z=\widehat\varpi+\widehat u_x,
		\]
		and abbreviate
		\[
		H=H(x,z),\qquad
		\widehat H=H(x,\widehat z),\qquad
		H_p=D_pH(x,z),\qquad
		\widehat H_p=D_pH(x,\widehat z).
		\]
		Set
		\[
		\delta\omega
		=
		\omega-\widehat\omega
		=
		(\delta m,\delta u,\delta\varpi)
		\in\mathcal X.
		\]
		Evaluating the difference against \(\delta\omega\), the fixed
		initial-data and supply terms cancel, as do the temporal terms
		\(\delta u_t\,\delta m-\delta m\,\delta u_t\).
		 Since
		\(\delta u_x+\delta\varpi=z-\widehat z\), we obtain
		\begin{align}
			\langle A(\omega)-A(\widehat\omega),\omega-\widehat\omega\rangle
			&=
			\int_{\Omega_T}
			\bigl[
			-(H-\widehat H)\delta m
			+(mH_p-\widehat m\widehat H_p)(z-\widehat z)
			\bigr]\,dx\,dt
			\notag\\
			&=
			\int_{\Omega_T}
			m\bigl[
			H_p(z-\widehat z)-(H-\widehat H)
			\bigr]\,dx\,dt
			\notag\\
			&\quad
			+
			\int_{\Omega_T}
			\widehat m\bigl[
			(H-\widehat H)-\widehat H_p(z-\widehat z)
			\bigr]\,dx\,dt .
			\label{eq:cont-convex-split}
		\end{align}
		By convexity of \(p\mapsto H(x,p)\),
		\[
		H-\widehat H
		\le H_p(z-\widehat z),
		\qquad
		H-\widehat H
		\ge \widehat H_p(z-\widehat z).
		\]
		Hence both brackets in~\eqref{eq:cont-convex-split} are nonnegative,
		and \eqref{eq:cont-mono} follows from \(m,\widehat m\ge0\). In fact, uniform convexity of $H$ in $p$ gives
		\[
		\langle A(\omega)-A(\widehat\omega),
		\omega-\widehat\omega\rangle
		\ge
		\frac{\theta}{2}
		\int_{\Omega_T}(m+\widehat m)|z-\widehat z|^2\,dx\,dt.
		\]
	\end{proof}
	
	\begin{remark}[Saddle-point interpretation]
		\label{rem:saddle-point}
		For comparison with the discrete construction, formally define
		\begin{equation}\label{eq:saddle-functional}
			\mathcal J(m,u,\varpi)
			=
			\int_0^T\!\!\int_\Omega
			\bigl(H(x,\varpi(t)+u_x)-u_t\bigr)m\,dx\,dt
			-\int_\Omega u(x,0)\bar m(x)\,dx
			+\int_0^T Q(t)\varpi(t)\,dt .
		\end{equation}
		Where these variations are justified, differentiating
		\eqref{eq:saddle-functional} shows that,
		with the ordering \(\omega=(m,u,\varpi)\), \(A\) is the
		saddle-gradient
		\begin{equation}\label{eq:saddle-gradient}
			A(m,u,\varpi)
			=
			\bigl(
			-D_m\mathcal J(m,u,\varpi),\,
			D_u\mathcal J(m,u,\varpi),\,
			D_\varpi\mathcal J(m,u,\varpi)
			\bigr).
		\end{equation}
		
		On such a domain,
		\(\mathcal J\) is affine, hence concave, in \(m\) and, for \(m\ge0\),
		convex in \((u,\varpi)\), which is the convex--concave structure of
		Lemma~\ref{lem:saddle-gradient-monotone}. The discrete operator of Section~\ref{sec:discrete} is built from the discrete analogue of~\eqref{eq:saddle-functional}, where the same criterion is applied 
		in Proposition~\ref{prop:disc-mono}.
	\end{remark}

	\subsection{Function-space realization and localized admissible set}
	\label{sec:spaces-hilbert}
	
We use \(L^2\) norms for the density and price directions and a
space-time \(H^1\) norm for the value-function direction.
Pointwise bounds on the momentum make the nonlinear terms in
\eqref{eq:cont-A-weak} compatible with this choice.
Define
	\begin{equation}\label{eq:H1T}
		H^1_{\mathrm{term}}
		=
		\bigl\{
		\varphi\in H^1(\Omega_T):
		\varphi(\cdot,T)=0
		\bigr\},
	\end{equation}
	and take
	\begin{equation}\label{eq:X-hilbert-def}
		\mathcal X
		=
		L^2(\Omega_T)
		\times H^1_{\mathrm{term}}
		\times L^2(0,T).
	\end{equation}
	The \(H^1_{\mathrm{term}}\) factor controls
	\(\tilde u_t,\tilde u_x\), and the trace at \(t=0\), while the density
	and price directions require only \(L^2\) control.
	
	We equip \(\mathcal X\) with the inner product
	\begin{align}
		(\eta,\tilde\eta)_{\mathcal X}
		&=
		\int_0^T\!\!\int_\Omega
		\psi\,\tilde\psi\,dx\,dt
		\notag\\
		&\quad+
		\int_0^T\!\!\int_\Omega
		\bigl(
		\phi_t\tilde\phi_t
		+\phi_x\tilde\phi_x
		+\phi\tilde\phi
		\bigr)\,dx\,dt
		+
		\int_0^T\chi\tilde\chi\,dt,
		\label{eq:X-ip}
	\end{align}
	for
	\(\eta=(\psi,\phi,\chi)\) and
	\(\tilde\eta=(\tilde\psi,\tilde\phi,\tilde\chi)\in\mathcal X\).
	
	With the fixed extension \(\bar u\) taken in \(H^1(\Omega_T)\),
	the affine state space becomes
	\[
	\mathcal A
	=
	L^2(\Omega_T)
	\times
	\bigl(\bar u+H^1_{\mathrm{term}}\bigr)
	\times L^2(0,T).
	\]

	For
	\(M=(M_m,M_g,M_\varpi)\in (0,\infty)^3\), define the localized
	admissible set
	\begin{equation}\label{eq:KM-cont}
		\mathcal K_M
		=
		\left\{
		\omega=(m,u,\varpi)\in\mathcal A:
		\begin{aligned}
			&0\le m\le M_m
			\quad\text{a.e. in }\Omega_T,\\
			&\|u_x\|_{L^\infty(\Omega_T)}\le M_g,\\
			&\|\varpi\|_{L^\infty(0,T)}\le M_\varpi
		\end{aligned}
		\right\}.
	\end{equation}
	The set \(\mathcal K_M\) is convex and closed in \(\mathcal A\) with
	the metric induced by \(\mathcal X\); the pointwise bounds pass to
	strong \(L^2\) limits after extracting almost-everywhere convergent
	subsequences. 
	The set need not be bounded in \(\mathcal X\), since no uniform bound
	on \(u_t\) is imposed.
	
	The estimates of Section~\ref{sec:formulation} motivate choosing
	\(M_\varpi,M_g,M_m\) to contain the continuous solution; analogous
	bounds for the discrete iterates are imposed or monitored separately.
	
	For every \(\omega\in\mathcal K_M\), the momentum \(p=\varpi+u_x\) satisfies
	\[
	\|p\|_{L^\infty(\Omega_T)}
	\le M_g+M_\varpi.
	\]
	Since \(H\) and \(D_pH\) are continuous on the compact set
	\[
	\overline\Omega
	\times[-M_g-M_\varpi,M_g+M_\varpi],
	\]
	there is \(C_M>0\) such that
	\[
	|H(x,p)|+|D_pH(x,p)|
	\le C_M
	\qquad
	\text{a.e. in }\Omega_T.
	\]
	Consequently,
	\[
	u_t-H(x,p)\in L^2(\Omega_T),
	\qquad
	m,\;mD_pH(x,p)\in L^2(\Omega_T),
	\]
	so the HJB and transport components are bounded on
	\(L^2(\Omega_T)\) and \(H^1_{\mathrm{term}}\), respectively.
	
	Since \(\bar m\in L^2(\Omega)\) and the trace map
	\(H^1(\Omega_T)\to L^2(\Omega)\) is continuous, the initial-data term
	is bounded.
	Finally,
	\[
	\left\|
	\int_\Omega
	m(x,t)D_pH\bigl(x,p(x,t)\bigr)\,dx
	\right\|_{L^2(0,T)}
	\le
	C_M|\Omega|^{1/2}
	\|m\|_{L^2(\Omega_T)}.
	\]
	Since \(Q\in L^2(0,T)\) by~\ref{H:data}, the market-clearing residual
	belongs to \(L^2(0,T)\). 
	Hence \(A:\mathcal K_M\to\mathcal X'\) is well defined and, by
	Proposition~\ref{prop:cont-mono}, monotone. The upper bound on \(m\) is not needed for well-definedness;
	its discrete analogue is used in the Lipschitz estimate of
	Section~\ref{sec:lip}.
	
	\subsection{Riesz representative}
	\label{sec:hilbert-riesz}
	
	For \(\omega=(m,u,\varpi)\in\mathcal K_M\), write the three components
	of \(A(\omega)\) as
	\[
	A(\omega)
	=
	\bigl(
	R_{\mathrm{HJB}}(\omega),
	R_{\mathrm{tr}}(\omega),
	R_{\mathrm{mc}}(\omega)
	\bigr),
	\]
	where
	\begin{equation}\label{eq:cont-R-HJB}
		R_{\mathrm{HJB}}(\omega)
		=
		u_t-H(x,\varpi+u_x)
		\in L^2(\Omega_T),
	\end{equation}
	the transport component is the functional
	\(R_{\mathrm{tr}}(\omega)\in(H^1_{\mathrm{term}})'\) defined by
	\begin{align}
		\langle R_{\mathrm{tr}}(\omega),\tilde u\rangle
		&=
		\int_0^T\!\!\int_\Omega
		m\bigl(
		-\tilde u_t
		+D_pH(x,\varpi+u_x)\tilde u_x
		\bigr)\,dx\,dt
		\notag\\
		&\quad
		-\int_\Omega
		\bar m(x)\tilde u(x,0)\,dx,
		\qquad
		\tilde u\in H^1_{\mathrm{term}},
		\label{eq:cont-R-tr}
	\end{align}
	and
	\begin{equation}\label{eq:cont-R-mc}
		R_{\mathrm{mc}}(\omega)(t)
		=
		\int_\Omega
		m(x,t)D_pH\bigl(x,\varpi(t)+u_x(x,t)\bigr)\,dx
		+Q(t)
		\in L^2(0,T).
	\end{equation}
	
	The Riesz operator
	\(\mathcal R:\mathcal X'\to\mathcal X\) associated with
	\eqref{eq:X-ip} has the block-diagonal form
	\begin{equation}\label{eq:cont-Riesz-block}
		\mathcal R
		=
		\operatorname{diag}\!\bigl(
		I,\mathcal R_u,I
		\bigr),
	\end{equation}
	where
	\[
	\mathcal R_u:
	(H^1_{\mathrm{term}})'
	\longrightarrow
	H^1_{\mathrm{term}}
	\]
	is the Riesz operator associated with the value function in
	\eqref{eq:X-ip}.
	
	Consequently,
	\begin{equation}\label{eq:cont-RA-components}
		\mathcal RA(\omega)
		=
		\bigl(
		R_{\mathrm{HJB}}(\omega),
		\Phi(\omega),
		R_{\mathrm{mc}}(\omega)
		\bigr),
	\end{equation}
	where
	\(\Phi(\omega)=\mathcal R_uR_{\mathrm{tr}}(\omega)
	\in H^1_{\mathrm{term}}\) is characterized by
	\begin{equation}\label{eq:Phi-weak}
		\int_0^T\!\!\int_\Omega
		\bigl(
		\Phi_t\,\tilde u_t
		+\Phi_x\,\tilde u_x
		+\Phi\,\tilde u
		\bigr)\,dx\,dt
		=
		\langle R_{\mathrm{tr}}(\omega),\tilde u\rangle
		\qquad
		\forall\,\tilde u\in H^1_{\mathrm{term}}.
	\end{equation}
	Thus the HJB and market-clearing residuals are represented directly in
	the \(L^2\) density and price norms, while the transport residual is
	lifted through the space-time \(H^1\) value-function norm.
	
	Formally, for smooth \(m\) and
	\(mD_pH(x,\varpi+u_x)\) with zero spatial flux, \(\Phi\) solves
	\[
	-\Phi_{tt}-\Phi_{xx}+\Phi
	=
	m_t-\bigl(mD_pH(x,\varpi+u_x)\bigr)_x
	\]
	with
	\[
	\Phi(\cdot,T)=0,
	\qquad
	\Phi_t(\cdot,0)=\bar m-m(\cdot,0),
	\qquad
	\partial_n\Phi=0
	\ \text{on }\partial\Omega\times(0,T).
	\]
	
	For general \(\omega\in\mathcal K_M\), \eqref{eq:Phi-weak}
	remains the definition: the zero-flux assumption is not imposed
	in \(\mathcal K_M\), and an \(L^2\) density need not have an
	initial trace.
	By Proposition~\ref{prop:cont-mono} and
	\eqref{eq:abstract-lift-identity}, \(\mathcal R A\) is monotone on
	\(\mathcal K_M\).
	
	\subsection{Failure of cocoercivity}
	\label{sec:cocoercivity}

For \(\beta>0\), \(\mathcal R A\) is \(\beta\)-cocoercive on
\(\mathcal K_M\) if
\[
\bigl(\mathcal{R}A(\omega)-\mathcal{R}A(\widehat\omega),\;
\omega-\widehat\omega\bigr)_{\mathcal{X}}
\ge
\beta\,
\bigl\|\mathcal{R}A(\omega)-\mathcal{R}A(\widehat\omega)
\bigr\|_{\mathcal{X}}^{2}
\qquad
\forall\,\omega,\widehat\omega\in\mathcal K_M.
\]
	
	If an operator is \(\beta\)-cocoercive on a nonempty closed
	convex feasible set and the associated VI has a solution,
	then the projected-forward iteration converges weakly to
	a VI solution for every fixed step size
	\(\nu\in(0,2\beta)\) \cite{bauschke2017convex}.
	Monotonicity alone does not give this guarantee.
	We show that cocoercivity fails in every
	equivalent Hilbert metric due to the affine dependence on the density:
	a density-only variation can alter the transport or market-clearing residual
	while leaving the momentum, the HJB residual, and consequently the
	monotonicity pairing unchanged.

	\begin{proposition}[Density-only obstruction to cocoercivity]
		\label{prop:cont-cocoerc}
		If \(\mathcal K_M\ne\varnothing\), then, for every equivalent
		Hilbert inner product \((\cdot,\cdot)_*\) on \(\mathcal X\),
		the associated Riesz representative \(\mathcal R_*A\) is not
		cocoercive on \(\mathcal K_M\).
	\end{proposition}
	
	\begin{proof}
		Choose feasible \(u,\varpi\). By the product structure of
		\(\mathcal K_M\), the states
		\[
		\omega=(M_m/2,u,\varpi),
		\qquad
		\widehat\omega=(0,u,\varpi)
		\]
		are feasible. Set \(\delta A=A(\omega)-A(\widehat\omega)\).
		Since the HJB residual is independent of \(m\),
		\[
		\langle\delta A,\omega-\widehat\omega\rangle=0.
		\]
		On the other hand, testing the transport component with
		\(\tilde u(x,t)=T-t\) gives
		\[
		\langle\delta A,(0,T-t,0)\rangle
		=\frac{M_m}{2}|\Omega|T>0,
		\]
		so \(\delta A\ne0\).
		
		For any equivalent Hilbert inner product, the Riesz identity
		and injectivity of \(\mathcal R_*\) therefore give
		\[
		(\mathcal R_*\delta A,\omega-\widehat\omega)_*=0,
		\qquad
		\|\mathcal R_*\delta A\|_*>0.
		\]
		This contradicts cocoercivity for every \(\beta>0\).
	\end{proof}

	\section{Discrete Residual Operator}
	\label{sec:discrete}

	From this point on, we restrict to the quadratic Hamiltonian
	\[
	H(x,p)=\frac{p^2}{2}.
	\]
	The residual is derived from a discrete saddle functional, with the
	transport operator defined as the algebraic adjoint of the linearized
	monotone upwind Hamiltonian. Throughout, the subscript \(\Delta\) denotes
	dependence on the grid \((\Delta x,\Delta t)\).
	
	\subsection{Grid and degrees of freedom}
	\label{sec:grid}
	
	Let \([a,b]\) be chosen as in the truncation argument of
	Section~\ref{sec:truncation-price}. We use the uniform grid
	\[
	\Dx=\frac{b-a}{I},\qquad
	\Dt=\frac{T}{J},\qquad
	x_i=a+i\Dx,\qquad
	t_j=j\Dt,
	\]
	with \(i=0,\dots,I\) and \(j=0,\dots,J\). 
	The terminal value is prescribed by
	\[
	u_{i,J}=\bar u_i:=\bar u(x_i),
	\qquad
	i=0,\dots,I.
	\]
	We also choose nonnegative discrete initial-density data
	\((\bar m_i)_{i=0}^{I}\), consistently approximating \(\bar m\) and
	normalized with respect to the discrete spatial pairing:
	\[
	\bar m_i\ge0,
	\qquad
	\Dx\sum_{i=0}^{I}\bar m_i=1.
	\]
	The values \(\bar m_i\) are data and
	enter through the initial boundary term in the discrete saddle
	functional.
	
	The discrete unknowns are
	\[
	\begin{aligned}
		&u_{i,j},
		&&i=0,\dots,I,\quad j=0,\dots,J-1,\\
		&m_{i,j},
		&&i=0,\dots,I,\quad j=1,\dots,J,\\
		&\varpi_j,
		&&j=0,\dots,J-1.
	\end{aligned}
	\]
	Let \(\Xh\) be the corresponding direction space of triples
	\(\tilde\omega_\Delta=(\tilde m,\tilde u,\tilde\varpi)\) on these free
	degrees of freedom. We set \(\tilde u_{\cdot,J}=0\); there is no density
	direction at \(j=0\).
	
	Choose any grid extension of the terminal cost, still denoted by
	\(\bar u\), with \(\bar u_{i,J}=\bar u_i\), and set
	\[
	\bar\omega_\Delta=(0,\bar u,0).
	\]
	The affine state space is \(\bar\omega_\Delta+\Xh\).
	
	\subsection{Monotone upwind Hamiltonian}
	\label{sec:ham}
	
	For \(j=0,\dots,J-1\), define
	\[
	D_+u_{i,j}
	=
	\frac{u_{i+1,j}-u_{i,j}}{\Dx},
	\qquad
	0\le i\le I-1,
	\]
	and
	\[
	D_-u_{i,j}
	=
	\frac{u_{i,j}-u_{i-1,j}}{\Dx},
	\qquad
	1\le i\le I.
	\]
	We use the monotone upwind numerical Hamiltonian
	\begin{equation}\label{eq:godunov}
		F(p,q,w)
		=
		\frac12\bigl(\max\{p-w,0\}\bigr)^2
		+
		\frac12\bigl(\max\{q+w,0\}\bigr)^2 .
	\end{equation}
	The map \(F\) is continuously differentiable and jointly convex in
	\((p,q,w)\). It is consistent with the quadratic Hamiltonian, since
	\[
	F(-s,s,w)=\frac12(s+w)^2
	\qquad
	\text{for all }s,w\in\R.
	\]
	Its partial derivatives are
	\begin{gather*}
		\partial_pF(p,q,w)=\max\{p-w,0\},
		\qquad
		\partial_qF(p,q,w)=\max\{q+w,0\},
		\\
		\partial_wF(p,q,w)=-\max\{p-w,0\}+\max\{q+w,0\}.
	\end{gather*}
	At the interior nodes, define the discrete Hamilton--Jacobi operator by
	\begin{equation}\label{eq:disc-HJB}
		G_{i,j}(u,\varpi)
		=
		F\bigl(
		-D_+u_{i,j},
		D_-u_{i,j},
		\varpi_j
		\bigr),
		\qquad
		1\le i\le I-1,\quad j=0,\dots,J-1.
	\end{equation}
	At the spatial boundary, we use the reduced numerical Hamiltonians
	obtained by retaining only the available spatial difference:
	\[
	G_{0,j}(u,\varpi)
	=
	\frac12
	\bigl(
	\max\{-D_+u_{0,j}-\varpi_j,0\}
	\bigr)^2,
	\]
	and
	\[
	G_{I,j}(u,\varpi)
	=
	\frac12
	\bigl(
	\max\{D_-u_{I,j}+\varpi_j,0\}
	\bigr)^2.
	\]
	These \(C^1\), convex reductions define the boundary derivatives of
	\(G\). At \(x=a\) and \(x=b\), they keep only the one-sided term whose
	feedback velocity points into \(\Omega\), and thus provide the
	state-constraint closure of the HJB discretization.

	\subsection{Discrete saddle functional and residuals}
	\label{sec:disc-residuals}
	
	Let \(Q_j=Q(t_j)\). For
	\(\omega_\Delta=(m,u,\varpi)\in \bar\omega_\Delta+X_\Delta\), define the
	discrete saddle functional
	\begin{equation}\label{eq:disc-saddle-functional}
		\mathcal J_\Delta(m,u,\varpi)
		=
		\Dx\Dt
		\sum_{j=0}^{J-1}\sum_{i=0}^{I}
		m_{i,j+1}
		\left[
		G_{i,j}(u,\varpi)
		-
		\frac{u_{i,j+1}-u_{i,j}}{\Dt}
		\right]
		-\Dx\sum_{i=0}^{I}\bar m_i u_{i,0}
		+
		\Dt\sum_{j=0}^{J-1}Q_j\varpi_j .
	\end{equation}
	We write \(m_{i,0}:=\bar m_i\); these values are data, not unknowns.
	The right-endpoint weight \(m_{\cdot,j+1}\) produces the forward time
	difference in the transport residual.
	
	We define the discrete residual \(A_\Delta\) as the saddle-gradient
	\begin{equation}\label{eq:Ad-saddle-gradient}
		A_\Delta(m,u,\varpi)
		=
		\bigl(
		-D_m\mathcal J_\Delta(m,u,\varpi),\,
		D_u\mathcal J_\Delta(m,u,\varpi),\,
		D_\varpi\mathcal J_\Delta(m,u,\varpi)
		\bigr).
	\end{equation}
	The derivatives are covectors; the residual arrays below represent them in the weighted duality pairing~\eqref{eq:disc-duality}. For
	\(\tilde\omega_\Delta=(\tilde m,\tilde u,\tilde\varpi)\in\Xh\), its
	dual action is
	\begin{align}
		\langle A_\Delta(\omega_\Delta),\tilde\omega_\Delta\rangle_\Delta
		&=
		\Dx\Dt\sum_{j=0}^{J-1}\sum_{i=0}^{I}
		R_{\Delta,\mathrm{HJB}}(\omega_\Delta)_{i,j}\,
		\tilde m_{i,j+1}
		\notag\\
		&\quad+
		\Dx\Dt\sum_{j=0}^{J-1}\sum_{i=0}^{I}
		R_{\Delta,\mathrm{tr}}(\omega_\Delta)_{i,j}\,
		\tilde u_{i,j}
		+
		\Dt\sum_{j=0}^{J-1}
		R_{\Delta,\mathrm{mc}}(\omega_\Delta)_j\,
		\tilde\varpi_j .
		\label{eq:disc-duality}
	\end{align}
	
	The density component gives the HJB residual
	\begin{equation}\label{eq:disc-Ru}
		R_{\Delta,\mathrm{HJB}}(\omega_\Delta)_{i,j}
		=
		\frac{u_{i,j+1}-u_{i,j}}{\Dt}
		-
		G_{i,j}(u,\varpi),
		\qquad
		i=0,\dots,I,\quad j=0,\dots,J-1.
	\end{equation}
	
	To write the value-function component, let \(L_j[u,\varpi]\) be the derivative of \(G_{\cdot,j}\) with respect
	to \(u_{\cdot,j}\). At interior nodes,
	\begin{equation}\label{eq:G-linearization}
		L_j[u,\varpi](v)_i
		=
		-\partial_pF_{i,j}\,D_+v_i
		+
		\partial_qF_{i,j}\,D_-v_i,
		\qquad 1\le i\le I-1,
	\end{equation}
	where
	\[
	\partial_pF_{i,j}
	=
	\partial_pF(-D_+u_{i,j},D_-u_{i,j},\varpi_j),
	\qquad
	\partial_qF_{i,j}
	=
	\partial_qF(-D_+u_{i,j},D_-u_{i,j},\varpi_j).
	\]
	At the boundary, set
	\[
	\partial_pF_{0,j}
	=
	\max\{-D_+u_{0,j}-\varpi_j,0\},
	\qquad
	\partial_qF_{I,j}
	=
	\max\{D_-u_{I,j}+\varpi_j,0\}.
	\]
	The reduced-Hamiltonian convention then gives
	\[
	L_j[u,\varpi](v)_0
	=
	-\partial_pF_{0,j}\,D_+v_0,
	\qquad
	L_j[u,\varpi](v)_I
	=
	\partial_qF_{I,j}\,D_-v_I.
	\]
	
	The algebraic adjoint \(L_j[u,\varpi]^*\) is defined by
	\begin{equation}\label{eq:adjoint-identity}
		\Dx\sum_{i=0}^{I}
		\rho_i\,L_j[u,\varpi](v)_i
		=
		\Dx\sum_{i=0}^{I}
		L_j[u,\varpi]^*(\rho)_i\,v_i
	\end{equation}
	for all spatial grid functions \(\rho\) and \(v\). Explicitly,
	\begin{equation}\label{eq:adj-transport}
		L_j[u,\varpi]^*(\rho)_i
		=
		\frac{1}{\Dx}\Bigl[
		\rho_i(\partial_pF_{i,j}+\partial_qF_{i,j})
		-\rho_{i-1}\partial_pF_{i-1,j}
		-\rho_{i+1}\partial_qF_{i+1,j}
		\Bigr],
		\qquad 1\le i\le I-1,
	\end{equation}
	with boundary rows
	\[
	L_j[u,\varpi]^*(\rho)_0
	=
	\frac{1}{\Dx}
	\Bigl[
	\rho_0\partial_pF_{0,j}
	-\rho_1\partial_qF_{1,j}
	\Bigr],
	\]
	and
	\[
	L_j[u,\varpi]^*(\rho)_I
	=
	\frac{1}{\Dx}
	\Bigl[
	\rho_I\partial_qF_{I,j}
	-\rho_{I-1}\partial_pF_{I-1,j}
	\Bigr].
	\]
	Since \(L_j[u,\varpi]\) annihilates constants,
	\begin{equation}\label{eq:Lstar-sum-zero}
		\sum_{i=0}^{I}L_j[u,\varpi]^*(\rho)_i=0.
	\end{equation}
	
	Differentiating \(\mathcal J_\Delta\) with respect to the free
	value-function entries, and using \(m_{\cdot,0}:=\bar m\), gives
	the transport residual
	\begin{equation}\label{eq:disc-Rm}
		R_{\Delta,\mathrm{tr}}(\omega_\Delta)_{i,j}
		=
		\frac{m_{i,j+1}-m_{i,j}}{\Dt}
		+
		L_j[u,\varpi]^*(m_{\cdot,j+1})_i,
		\qquad
		i=0,\dots,I,\quad j=0,\dots,J-1.
	\end{equation}
	By \eqref{eq:Lstar-sum-zero}, discrete mass is conserved whenever
	the transport residual vanishes.
	
	Finally, differentiating with respect to the price gives
	\begin{equation}\label{eq:disc-Rw}
		R_{\Delta,\mathrm{mc}}(\omega_\Delta)_j
		=
		\Dx\sum_{i=0}^{I}
		m_{i,j+1}\,\partial_\varpi G_{i,j}(u,\varpi)
		+
		Q_j,
		\qquad j=0,\dots,J-1.
	\end{equation}
	At interior nodes,
	\[
	\partial_\varpi G_{i,j}(u,\varpi)
	=
	\partial_wF\bigl(-D_+u_{i,j},D_-u_{i,j},\varpi_j\bigr),
	\]
	with derivatives of the reduced Hamiltonians used at \(i=0,I\).

	\subsection{Admissible sets}
	\label{sec:admissible}

	The basic admissible set is the convex subset of the affine state space

	\begin{equation}\label{eq:KDelta}
		K_\Delta
		=
		\Bigl\{
		\omega_\Delta\in\bar\omega_\Delta+\Xh:
		m_{i,j}\ge0
		\quad
		\forall\,i=0,\dots,I,\;\;j=1,\dots,J
		\Bigr\}.
	\end{equation}
	Only nonnegativity of the density is imposed at this level; no
	time-slice mass constraint is included in \(K_\Delta\).
	
	For the convergence argument, we use the bounded admissible set
	\begin{equation}\label{eq:KDeltaM}
		K_{\Delta,M}
		=
		\left\{
		\omega_\Delta\in K_\Delta:
		\begin{aligned}
			&0\le m_{i,j}\le M_m
			&&\forall\, i,\;\; j=1,\dots,J,\\
			&|u_{i,j}|\le M_u
			&&\forall\, i,\;\; j=0,\dots,J-1,\\
			&|\varpi_j|\le M_\varpi
			&&\forall\, j=0,\dots,J-1
		\end{aligned}
		\right\}.
	\end{equation}
	Here
	\(M=(M_m,M_u,M_\varpi)\in(0,\infty)^3\) is fixed independently of
	the grid. Note that, unlike the continuous triple in \eqref{eq:KM-cont}, whose middle entry bounds the slope, here the middle entry bounds the value function itself; a discrete slope bound enters separately through \(M_g\) below.
	
	For the mesh-independent Lipschitz estimate, we further restrict to
	\begin{equation}\label{eq:KDeltaMg}
		K_{\Delta,M,g}
		=
		\left\{
		\omega_\Delta\in K_{\Delta,M}:
		|D_+u_{i,j}|\le M_g
		\quad\text{for }i=0,\dots,I-1,\quad j=0,\dots,J-1
		\right\}.
	\end{equation}
	Since \(D_-u_{i,j}=D_+u_{i-1,j}\), this also bounds every backward difference. Here \(M_g\) is independent of \(\Dx\) and \(\Dt\). 
	Proposition~\ref{prop:lip} gives a mesh-independent Lipschitz
	bound on \(K_{\Delta,M,g}\) and a fixed-grid bound on
	\(K_{\Delta,M}\).

	\subsection{Discrete monotonicity}
	\label{sec:disc-mono}
	
	\begin{proposition}[Discrete monotonicity]
		\label{prop:disc-mono}
		\(A_\Delta\) is monotone on \(K_\Delta\) in the discrete duality
		pairing: for every
		\(\omega_\Delta,\widehat\omega_\Delta\in K_\Delta\),
		\begin{equation}\label{eq:disc-mono}
			\bigl\langle
			A_\Delta(\omega_\Delta)-A_\Delta(\widehat\omega_\Delta),
			\omega_\Delta-\widehat\omega_\Delta
			\bigr\rangle_\Delta
			\ge 0 .
		\end{equation}
	\end{proposition}
	
	\begin{proof}
		Each numerical Hamiltonian \(G_{i,j}\) is \(C^1\) and convex in
		\((u_{\cdot,j},\varpi_j)\), including at the boundary nodes.
		After fixing the terminal data, \(K_\Delta\)
		is the product of the
		nonnegative density cone with the value-function and price spaces. For
		\(m\ge0\), \(\mathcal J_\Delta\) is convex in \((u,\varpi)\), and affine, hence concave, in \(m\). Therefore
		\(A_\Delta=(-D_m\mathcal J_\Delta,D_u\mathcal J_\Delta,
		D_\varpi\mathcal J_\Delta)\) is monotone by
		Lemma~\ref{lem:saddle-gradient-monotone}.
	\end{proof}
	
	Analogous results apply to \(H(x,p)=p^2/2-V(x)\)
	if \(G_{i,j}\) is replaced by \(G_{i,j}-V(x_i)\).
	This changes the HJB residual by a fixed term and leaves
	the monotonicity and Lipschitz estimates in the upcoming section unchanged.
		
	\section{Discrete Riesz Operator and Lipschitz Estimate}
	\label{sec:analysis}
	
	We equip \(X_\Delta\) with a mixed \(L^2\)--Sobolev metric,
	identify the corresponding Riesz representative of \(A_\Delta\), and
	prove a mesh-independent Lipschitz estimate on
	\(K_{\Delta,M,g}\).
	
	\subsection{Discrete inner product and Riesz representative}
	\label{sec:disc-riesz}
	
	On the direction space \(X_\Delta\) of
	Section~\ref{sec:grid}, we use weighted discrete \(L^2\) norms in the
	density and price directions and a space-time \(H^1\) norm in the
	value-function direction, with \(\phi_{\cdot,J}=0\).
	
	For the density, set
	\begin{equation}\label{eq:norm-m}
		\|\psi\|_{m,\Delta}^2
		=
		\Dx\Dt
		\sum_{j=1}^{J}\sum_{i=0}^{I}
		|\psi_{i,j}|^2 .
	\end{equation}
	For the value function, set
	\begin{equation}\label{eq:norm-u}
		\|\phi\|_{u,\Delta}^2
		=
		\Dx\Dt\sum_{j=1}^{J}\sum_{i=0}^{I}
		\left|\frac{\phi_{i,j}-\phi_{i,j-1}}{\Dt}\right|^2
		+
		\Dx\Dt\sum_{j=0}^{J-1}\sum_{i=0}^{I}|\phi_{i,j}|^2
		+
		\Dx\Dt\sum_{j=0}^{J-1}\sum_{i=0}^{I-1}
		\left|\frac{\phi_{i+1,j}-\phi_{i,j}}{\Dx}\right|^2 .
	\end{equation}
	For the price, set
	\begin{equation}\label{eq:norm-w}
		\|\chi\|_{\varpi,\Delta}^2
		=
		\Dt\sum_{j=0}^{J-1}|\chi_j|^2 .
	\end{equation}
	The Hilbert inner product on \(X_\Delta\) is 
	\begin{equation}\label{eq:disc-ip}
		(\omega_\Delta,\tilde\omega_\Delta)_{X_\Delta}
		=
		(m,\tilde m)_{m,\Delta}
		+
		(u,\tilde u)_{u,\Delta}
		+
		(\varpi,\tilde\varpi)_{\varpi,\Delta},
	\end{equation}
	where the three inner products correspond to the norms
	\eqref{eq:norm-m}--\eqref{eq:norm-w} and
	\(\|\omega_\Delta\|_{X_\Delta}^2=(\omega_\Delta,\omega_\Delta)_{X_\Delta}\).
	
	The density and price directions have identity Riesz maps. Let
	\(\Ltt{T}+\Lxx\) be the symmetric positive-definite operator induced by
	the value-function norm, namely
	\[
	\Dx\Dt\sum_{j=0}^{J-1}\sum_{i=0}^{I}
	\bigl((\Ltt{T}+\Lxx)\phi\bigr)_{i,j}\tilde\phi_{i,j}
	=
	(\phi,\tilde\phi)_{u,\Delta}
	\]
	for all value directions \(\phi,\tilde\phi\). 
	Equivalently, with \(\phi_{i,J}=0\), these operators are
	\[
	(\Ltt{T}\phi)_{i,j}
	=
	\begin{cases}
		\displaystyle
		\phi_{i,0}
		+\frac{\phi_{i,0}-\phi_{i,1}}{\Dt^2},
		& j=0,\\[3mm]
		\displaystyle
		\phi_{i,j}
		+\frac{2\phi_{i,j}-\phi_{i,j-1}-\phi_{i,j+1}}{\Dt^2},
		& 1\le j\le J-1,
	\end{cases}
	\]
	for \(i=0,\dots,I\), and
	\[
	(\Lxx\phi)_{i,j}
	=
	\begin{cases}
		\displaystyle
		\frac{\phi_{0,j}-\phi_{1,j}}{\Dx^2},
		& i=0,\\[3mm]
		\displaystyle
		\frac{2\phi_{i,j}-\phi_{i-1,j}-\phi_{i+1,j}}{\Dx^2},
		& 1\le i\le I-1,\\[3mm]
		\displaystyle
		\frac{\phi_{I,j}-\phi_{I-1,j}}{\Dx^2},
		& i=I,
	\end{cases}
	\]
	for \(j=0,\dots,J-1\).
	Thus \(\Ltt{T}\) includes the value-function \(L^2\) term and the
	natural temporal boundary row at \(j=0\); \(\Lxx\) has
	homogeneous Neumann boundary rows.
	
	Therefore, the Riesz operator associated with \eqref{eq:disc-ip} is
	the block-diagonal map
	\(\mathcal{R}_\Delta:X_\Delta'\to X_\Delta\),
	\begin{equation}\label{eq:Rh-def}
		\mathcal{R}_\Delta
		=\operatorname{diag}\!\bigl(
		I,\,
		(\Ltt{T}+\Lxx)^{-1},\,
		I
		\bigr),
	\end{equation}
	characterized by
	\begin{equation}\label{eq:riesz-prop}
		(\mathcal{R}_\Delta f,\tilde\omega_\Delta)_{X_\Delta}
		=\langle f,\tilde\omega_\Delta\rangle_\Delta
		\qquad
		\forall\, f\in X_\Delta',\ \forall\,\tilde\omega_\Delta\in X_\Delta ,
	\end{equation}
	where \(\langle\cdot,\cdot\rangle_\Delta\) is the duality pairing
	defined in \eqref{eq:disc-duality}.
	
	Applying \(\mathcal{R}_\Delta\) to the discrete residual \(A_\Delta\) of
	Section~\ref{sec:discrete} gives the Riesz representative
	\begin{equation}\label{eq:RhAh-def}
		\mathcal{R}_\Delta A_\Delta(\omega_\Delta)
		=\bigl(
		R_{\Delta,\mathrm{HJB}},\,
		(\Ltt{T}+\Lxx)^{-1}R_{\Delta,\mathrm{tr}},\,
		R_{\Delta,\mathrm{mc}}
		\bigr).
	\end{equation}
	The assignment follows from \eqref{eq:disc-duality}. In
	particular, the HJB residual is identified with the density levels
	through the shift \(j\mapsto j+1\), while the transport and
	market-clearing residuals retain the value-function and price
	indexing.
	
	By Proposition~\ref{prop:disc-mono} and
	\eqref{eq:abstract-lift-identity},
	\(\mathcal R_\Delta A_\Delta\) is monotone on \(K_\Delta\) in the
	metric~\eqref{eq:disc-ip}. 
	The metric affects the Riesz representative, its Lipschitz constant,
	and the associated projection, but not the underlying
	duality-based monotonicity.

	\subsection{Lipschitz continuity on \(K_{\Delta,M,g}\)}
	\label{sec:lip}
	
	By the Riesz isometry
	\(\|\mathcal R_\Delta f\|_{X_\Delta}=\|f\|_{X_\Delta'}\),
	it suffices to estimate the three residuals in the dual product norm.
	We retain explicit constants to justify the step-size choice in
	Section~\ref{sec:num-iterations}.
	The
	HJB time difference is controlled by the value-function norm;
	summation by parts and \eqref{eq:adjoint-identity} transfer the temporal
	and spatial differences in the transport residual to its test
	direction.
	
	On \(K_{\Delta,M,g}\), the density, slope, and price bounds are
	independent of the mesh. We use the discrete spatial norm
	\[
	\|v\|_{\ell_x^2}^2
	=
	\Dx\sum_i|v_i|^2.
	\]
	If \(0\le\lambda_i\le M_m\), then, since
	\(\Dx\sum_i 1=(I+1)\Dx\le2|\Omega|\),
	\begin{equation}\label{eq:box-estimates}
		\begin{aligned}
			\Dx\sum_i\lambda_i|v_i|^2
			&\le M_m\|v\|_{\ell_x^2}^2,\\
			\Dx\sum_i\lambda_i|v_i|
			&\le M_m\sqrt{2|\Omega|}\,\|v\|_{\ell_x^2}.
		\end{aligned}
	\end{equation}

	 The product structure gives
	\begin{equation}\label{eq:dual-norm-split}
		\|A_\Delta(\omega_\Delta)-A_\Delta(\widehat\omega_\Delta)\|_{X_\Delta'}^2
		=
		\|\delta R_{\Delta,\mathrm{HJB}}\|_{(m,\Delta)'}^2
		+
		\|\delta R_{\Delta,\mathrm{tr}}\|_{(u,\Delta)'}^2
		+
		\|\delta R_{\Delta,\mathrm{mc}}\|_{\varpi,\Delta}^2,
	\end{equation}
	where \(\delta R_{\Delta,\bullet}=R_{\Delta,\bullet}(\omega_\Delta)
	-R_{\Delta,\bullet}(\widehat\omega_\Delta)\), and \(\|\cdot\|_{(m,\Delta)'}\),
	\(\|\cdot\|_{(u,\Delta)'}\) are the dual norms of \(\|\cdot\|_{m,\Delta}\)
	and \(\|\cdot\|_{u,\Delta}\). For the density, \(\|\cdot\|_{(m,\Delta)'}\) is the discrete
	space-time \(L^2\) norm, with the index shift \(j\mapsto j+1\).
	
	\begin{proposition}[Mesh-independent Lipschitz continuity]
		\label{prop:lip}
		Let \(K_{\Delta,M,g}\) be the nonempty refined admissible set
		\eqref{eq:KDeltaMg}, and let \(\mathcal R_\Delta\) be the mixed \(L^2\)--Sobolev Riesz
		map~\eqref{eq:Rh-def}. Then
		\begin{equation}\label{eq:lip}
			\|\mathcal{R}_\Delta A_\Delta(\omega_\Delta)
			-\mathcal{R}_\Delta A_\Delta(\widehat\omega_\Delta)\|_{X_\Delta}
			\le
			L_\Delta\,\|\omega_\Delta-\widehat\omega_\Delta\|_{X_\Delta}
			\qquad
			\forall\,\omega_\Delta,\widehat\omega_\Delta\in K_{\Delta,M,g},
		\end{equation}
		where \(L_\Delta\le L\) for a constant
		\(L=L(M_m,M_g,M_\varpi,|\Omega|)\) independent of
		\(\Dx\) and \(\Dt\).
		On \(K_{\Delta,M}\), \(\mathcal R_\Delta A_\Delta\) is Lipschitz on
		each fixed grid. 
	\end{proposition}
	\begin{proof}
		Write
		\[
		\omega_\Delta=(m,u,\varpi),
		\qquad
		\widehat\omega_\Delta
		=
		(\widehat m,\widehat u,\widehat\varpi),
		\]
		and set
		\[
		\delta m=m-\widehat m,\qquad
		\delta u=u-\widehat u,\qquad
		\delta\varpi=\varpi-\widehat\varpi.
		\]
		For all nodes, write
		\[
		g_{i,j}=G_{i,j}(u,\varpi),
		\qquad
		\widehat g_{i,j}
		=
		G_{i,j}(\widehat u,\widehat\varpi),
		\]
		with \(\partial_p\widehat F_{i,j}\), \(\partial_q\widehat F_{i,j}\)
		evaluated at \((\widehat u,\widehat\varpi)\), and derivatives at
		\(i=0,I\) taken from the reduced Hamiltonians.
		
		Since both states lie in \(K_{\Delta,M,g}\), all available
		positive-part arguments satisfy
		\[
		|-D_+u_{i,j}-\varpi_j|,
		\quad
		|D_-u_{i,j}+\varpi_j|
		\le R,
		\qquad
		R=M_g+M_\varpi,
		\]
		and likewise for the hatted state. Consequently, for all available
		derivatives,
		\[
		|\partial_pF_{i,j}|,\ |\partial_qF_{i,j}|
		\le R,
		\qquad
		|\partial_\varpi G_{i,j}(u,\varpi)|
		\le R,
		\]
		with the same bounds at the hatted state; the last bound holds because
		\(\partial_wF\) is the difference of two numbers in \([0,R]\). 
		Because the positive part function is 1-Lipschitz, we have
		\begin{align}
			|\partial_pF_{i,j}-\partial_p\widehat F_{i,j}|
			&\le |D_+\delta u_{i,j}|+|\delta\varpi_j|,
			\notag\\
			|\partial_qF_{i,j}-\partial_q\widehat F_{i,j}|
			&\le |D_-\delta u_{i,j}|+|\delta\varpi_j|,
			\notag\\
			|g_{i,j}-\widehat g_{i,j}|
			&\le
			R\bigl(
			|D_+\delta u_{i,j}|+|D_-\delta u_{i,j}|
			+|\delta\varpi_j|
			\bigr),
			\label{eq:flux-difference-bounds}
		\end{align}
		since the derivative bounds hold along the segment joining the
		two states in the convex set \(K_{\Delta,M,g}\). At boundary nodes,
		unavailable differences are omitted.
		
	Write \(|v|_{t,\Delta}\) and \(|v|_{x,\Delta}\) for the
	temporal- and spatial-difference seminorms in
	\eqref{eq:norm-u}. The forward and backward spatial
	differences have the same weighted \(\ell^2\) norm.
	Hence, by Cauchy--Schwarz,
	\begin{equation}\label{eq:Ru-bound}
		\begin{aligned}
			\|\delta R_{\Delta,\mathrm{HJB}}\|_{(m,\Delta)'}
			&\le
			|\delta u|_{t,\Delta}
			+2R|\delta u|_{x,\Delta}
			+R\sqrt{2|\Omega|}\,
			\|\delta\varpi\|_{\varpi,\Delta}\\
			&\le
			\sqrt{1+4R^2}\,\|\delta u\|_{u,\Delta}
			+R\sqrt{2|\Omega|}\,
			\|\delta\varpi\|_{\varpi,\Delta}.
		\end{aligned}
	\end{equation}

		For the transport residual in the value dual norm, let \(\phi\) be a value direction with \(\phi_{\cdot,J}=0\) and
		\(\|\phi\|_{u,\Delta}\le1\). 
		Set \(\delta m_{\cdot,0}=0\) to encode cancellation of the
		common initial data. Discrete summation by parts, using
		\(\phi_{\cdot,J}=0\), gives
		\begin{equation}\label{eq:sbp-transport}
			\Dx\Dt\sum_{j=0}^{J-1}\sum_i
			\frac{\delta m_{i,j+1}-\delta m_{i,j}}{\Dt}\,\phi_{i,j}
			=
			-\Dx\Dt\sum_{j=1}^{J}\sum_i
			\delta m_{i,j}\,\frac{\phi_{i,j}-\phi_{i,j-1}}{\Dt},
		\end{equation}
		bounded in absolute value by \(\|\delta m\|_{m,\Delta}\)
		times the temporal-difference seminorm of \(\phi\) in
		\eqref{eq:norm-u}. For the
		adjoint transport term, the adjoint identity~\eqref{eq:adjoint-identity}
		transfers the spatial differences onto \(\phi\), giving at each time level
		\begin{align}
			&\Dx\sum_i\bigl[L_j[u,\varpi]^*(m_{\cdot,j+1})_i
			-L_j[\widehat u,\widehat\varpi]^*(\widehat m_{\cdot,j+1})_i\bigr]\phi_{i,j}
			\notag\\
			&=
			\Dx\sum_{i=0}^{I-1}
			\bigl[m_{i,j+1}\partial_pF_{i,j}
			-\widehat m_{i,j+1}\partial_p\widehat F_{i,j}\bigr]
			(-D_+\phi_{i,j})
			+
			\Dx\sum_{i=1}^{I}
			\bigl[m_{i,j+1}\partial_qF_{i,j}
			-\widehat m_{i,j+1}\partial_q\widehat F_{i,j}\bigr]
			D_-\phi_{i,j}.
			\label{eq:adj-diff}
		\end{align}
		Using
		\[
		m\partial_pF-\widehat m\,\partial_p\widehat F
		=
		\delta m\,\partial_pF
		+
		\widehat m(\partial_pF-\partial_p\widehat F),
		\]
		and the analogous identity for the \(q\)-term,
		the coefficient bounds, \eqref{eq:flux-difference-bounds}, and
		\eqref{eq:box-estimates} 
		give, after summing in time and applying Cauchy--Schwarz,
		the following upper bound for the absolute value of the
		transport pairing:
		\[
		\begin{aligned}
			&\|\delta m\|_{m,\Delta}
			\bigl(|\phi|_{t,\Delta}+2R|\phi|_{x,\Delta}\bigr)\\
			&\quad+
			\left(
			2M_m|\delta u|_{x,\Delta}
			+2M_m\sqrt{2|\Omega|}\,
			\|\delta\varpi\|_{\varpi,\Delta}
			\right)|\phi|_{x,\Delta}.
		\end{aligned}
		\]
		
		Since
		\(
		|\phi|_{t,\Delta}+2R|\phi|_{x,\Delta}
		\le\sqrt{1+4R^2}\,\|\phi\|_{u,\Delta},
		\)
		taking the supremum over \(\|\phi\|_{u,\Delta}\le1\) gives
		\begin{equation}\label{eq:Rm-bound}
			\|\delta R_{\Delta,\mathrm{tr}}\|_{(u,\Delta)'}
			\le
			\sqrt{1+4R^2}\,\|\delta m\|_{m,\Delta}
			+2M_m\|\delta u\|_{u,\Delta}
			+2M_m\sqrt{2|\Omega|}\,\|\delta\varpi\|_{\varpi,\Delta}.
		\end{equation}

		For the market-clearing residual in the price norm, \eqref{eq:disc-Rw} gives, for \(j=0,\dots,J-1\),
		\[
		(\delta R_{\Delta,\mathrm{mc}})_j
		=
		\Dx\sum_{i=0}^{I}
		\Bigl[
		\delta m_{i,j+1}\,
		\partial_\varpi G_{i,j}(u,\varpi)
		+
		\widehat m_{i,j+1}
		\bigl(
		\partial_\varpi G_{i,j}(u,\varpi)
		-
		\partial_\varpi G_{i,j}(\widehat u,\widehat\varpi)
		\bigr)
		\Bigr].
		\]
		Moreover,
		\[
		\left|
		\partial_\varpi G_{i,j}(u,\varpi)
		-
		\partial_\varpi G_{i,j}(\widehat u,\widehat\varpi)
		\right|
		\le
		|D_+\delta u_{i,j}|+|D_-\delta u_{i,j}|
		+2|\delta\varpi_j|,
		\]
		with unavailable boundary differences omitted.
		
		Using this estimate, the coefficient bounds, and
		\eqref{eq:box-estimates}, then summing in time, gives
		\begin{equation}\label{eq:Rw-bound}
			\|\delta R_{\Delta,\mathrm{mc}}\|_{\varpi,\Delta}
			\le
			R\sqrt{2|\Omega|}\,\|\delta m\|_{m,\Delta}
			+2M_m\sqrt{2|\Omega|}\,\|\delta u\|_{u,\Delta}
			+4M_m|\Omega|\,\|\delta\varpi\|_{\varpi,\Delta}.
		\end{equation}
		
		Combining \eqref{eq:Ru-bound}, \eqref{eq:Rm-bound}, and
		\eqref{eq:Rw-bound} in \eqref{eq:dual-norm-split}, and bounding each
		squared row by the sum of the squares of its coefficients times
		\(\|\omega_\Delta-\widehat\omega_\Delta\|_{X_\Delta}^2\), gives
		\[
		\|A_\Delta(\omega_\Delta)-A_\Delta(\widehat\omega_\Delta)\|_{X_\Delta'}
		\le
		L\,\|\omega_\Delta-\widehat\omega_\Delta\|_{X_\Delta},
		\]
		with
		\begin{equation}\label{eq:L-explicit}
		L^2
		=
		2(1+4R^2)+4R^2|\Omega|+4M_m^2+16M_m^2|\Omega|+16M_m^2|\Omega|^2,
		\qquad
		R=M_g+M_\varpi,
		\end{equation}
		independent of the mesh. The Riesz isometry then gives
		\eqref{eq:lip}.
		
		On \(K_{\Delta,M}\), every slope satisfies
		\(|D_\pm u_{i,j}|\le 2M_u/\Dx\), so
		\(K_{\Delta,M}=K_{\Delta,M,g}\) with \(M_g=2M_u/\Dx\), and the
		estimate above gives a Lipschitz constant on each fixed grid. 
	\end{proof}
	
	\section{Projected Extragradient Algorithm and Convergence}
	\label{sec:eg}
	
	We introduce the exact metric-projection extragradient
	method and prove its convergence for the discrete variational
	inequality. We then state the conditions under which the step size can
	be chosen uniformly under grid refinement.

	\subsection{Projected extragradient iteration}
	\label{sec:eg-general}
	
	Set
	\[
	F_\Delta:=\mathcal R_\Delta A_\Delta.
	\]
	By Proposition~\ref{prop:disc-mono} and the Riesz lift identity,
	\(F_\Delta\) is monotone. It is Lipschitz on \(K_{\Delta,M}\) on each
	fixed grid, while Proposition~\ref{prop:lip} gives a mesh-independent
	bound on \(K_{\Delta,M,g}\). The discrete variational inequality is to find
	\(\omega_\Delta^*\in K_{\Delta,M}\) such that
	\begin{equation}\label{eq:VI}
		(F_\Delta(\omega_\Delta^{*}),\,
		\tilde\omega_\Delta-\omega_\Delta^{*})_{X_\Delta} \ge 0
		\qquad
		\forall\,\tilde\omega_\Delta\in K_{\Delta,M}.
	\end{equation}

	The density-only argument of
	Proposition~\ref{prop:cont-cocoerc} also applies to \(A_\Delta\), so
	\(F_\Delta\) is not cocoercive on \(K_{\Delta,M}\).
	We
	therefore use Korpelevich's projected extragradient method
	\cite{korpelevich1976,facchinei2003finite}, which applies to monotone
	Lipschitz variational inequalities.
	
	For
	\(z_\Delta=(z_m,z_u,z_\varpi)\in\bar\omega_\Delta+X_\Delta\), define
	the metric projection by
	\begin{equation}\label{eq:exact-metric-projection}
		P_{K_{\Delta,M}}^{X_\Delta}(z_\Delta)
		=
		\operatorname*{argmin}_{\eta_\Delta\in K_{\Delta,M}}
		\frac12
		\|\eta_\Delta-z_\Delta\|_{X_\Delta}^2.
	\end{equation}
	Given \(\omega_\Delta^0\in K_{\Delta,M}\) and
	\(\nu>0\), define
	\begin{align}
		\hat\omega_\Delta^{k}
		&= P_{K_{\Delta,M}}^{X_\Delta}\bigl(
		\omega_\Delta^{k}-\nu\,F_\Delta(\omega_\Delta^{k})
		\bigr),
		\label{eq:eg-pred}\\[4pt]
		\omega_\Delta^{k+1}
		&= P_{K_{\Delta,M}}^{X_\Delta}\bigl(
		\omega_\Delta^{k}-\nu\,F_\Delta(\hat\omega_\Delta^{k})
		\bigr).
		\label{eq:eg-corr}
	\end{align}
	Each iteration uses two evaluations of \(F_\Delta\) and two
	projections and acts only on the free degrees of freedom, so the
	terminal value \(u_{\cdot,J}=\bar u\) remains fixed.
	The initial density enters as fixed data \(m_{\cdot,0}:=\bar m\)
	in the transport residual; only density levels \(j=1,\dots,J\)
	are updated.
	
	For the inner product~\eqref{eq:disc-ip}, the projection decomposes
	as follows:
	\[
	P_{K_{\Delta,M}}^{X_\Delta}(z_\Delta)
	=
	\bigl(
	\operatorname{clip}_{[0,M_m]}(z_m),\,
	\Pi_u(z_u),\,
	\operatorname{clip}_{[-M_\varpi,M_\varpi]}(z_\varpi)
	\bigr),
	\]
	where clipping is componentwise, $\operatorname{clip}_{[a,b]}(r)
	=
	\min\{\max\{r,a\},b\}$, and
	\begin{equation}\label{eq:value-metric-projection}
		\Pi_u(z_u)
		=
		\operatorname*{argmin}_{\substack{
				v_{\cdot,J}=\bar u\\
				|v_{i,j}|\le M_u,\;
				0\le i\le I,\;0\le j\le J-1}}
		\frac12\|v-z_u\|_{u,\Delta}^2 .
	\end{equation}
	Thus \(\Pi_u(z_u)\) is the unique minimizer of a strictly convex
	bound-constrained quadratic problem with quadratic form induced
	by \(\Ltt{T}+\Lxx\).
	Algorithm~1 summarizes the exact method.
	For the mesh-independent result, the projection is instead onto
	\(K_{\Delta,M,g}\); equivalently, the constraints
	\[
	|D_+v_{i,j}|\le M_g,
	\qquad i=0,\dots,I-1,\quad j=0,\dots,J-1
	\]
	are added to \eqref{eq:value-metric-projection}.

	\medskip
	\noindent\rule{\textwidth}{0.4pt}
	\noindent\textbf{Algorithm 1.}
	Exact metric-projection extragradient method.
	
	\noindent\rule{\textwidth}{0.4pt}
	\begin{enumerate}[leftmargin=2em]
		\item Choose
		\(\omega_\Delta^0\in K_{\Delta,M}\),
		\(\nu>0\), \(\mathrm{tol}>0\), and \(k_{\max}\).
		
		\item For \(k=0,\dots,k_{\max}-1\):
		\begin{enumerate}[label=\textup{(\alph*)},leftmargin=2em]
			\item Evaluate \(F_\Delta(\omega_\Delta^k)\) using
			\eqref{eq:disc-Ru}, \eqref{eq:disc-Rm},
			\eqref{eq:disc-Rw}, and \eqref{eq:RhAh-def}.
			\item Compute the predictor by \eqref{eq:eg-pred}.
			\item Evaluate \(F_\Delta(\hat\omega_\Delta^k)\).
			\item Compute the corrector by \eqref{eq:eg-corr}.
			\item Stop if
			\(\|\omega_\Delta^{k+1}-\omega_\Delta^k\|_{X_\Delta}
			<\mathrm{tol}\).
		\end{enumerate}
		
		\item Return the latest iterate.
	\end{enumerate}
	\noindent\rule{\textwidth}{0.4pt}
	\medskip
	
	\begin{remark}[Natural VI residual and stopping criterion]
		For \(\nu>0\) and \(\omega_\Delta\in K_{\Delta,M}\), define
		\[
		r_\nu(\omega_\Delta)
		:=
		\nu^{-1}
		\left\|\omega_\Delta-
		P_{K_{\Delta,M}}^{X_\Delta}
		\bigl(\omega_\Delta-\nu F_\Delta(\omega_\Delta)\bigr)
		\right\|_{X_\Delta}.
		\]
		By the characterization of the metric projection, this residual
		vanishes if and only if \(\omega_\Delta\) solves the associated VI.
		The predictor therefore provides
		\(r_\nu(\omega_\Delta^k)
		=\nu^{-1}\|\omega_\Delta^k-\hat\omega_\Delta^k\|_{X_\Delta}\).
		
		Let \(L_\Delta>0\) be a Lipschitz constant of \(F_\Delta\) on
		\(K_{\Delta,M}\), and suppose \(0<\nu L_\Delta<1\).
		Both projection steps in
		\eqref{eq:eg-pred}--\eqref{eq:eg-corr} have the same base point
		\(\omega_\Delta^k\). Hence nonexpansiveness of the metric
		projection and Lipschitz continuity of \(F_\Delta\) give
		\[
		\begin{aligned}
			\|\omega_\Delta^{k+1}-\hat\omega_\Delta^k\|_{X_\Delta}
			&\le
			\nu\|F_\Delta(\hat\omega_\Delta^k)
			-F_\Delta(\omega_\Delta^k)\|_{X_\Delta}\\
			&\le
			\nu L_\Delta
			\|\omega_\Delta^k-\hat\omega_\Delta^k\|_{X_\Delta}.
		\end{aligned}
		\]
		The reverse triangle inequality therefore yields
		\[
		\begin{aligned}
			\|\omega_\Delta^{k+1}-\omega_\Delta^k\|_{X_\Delta}
			&\ge
			\|\omega_\Delta^k-\hat\omega_\Delta^k\|_{X_\Delta}
			-\|\omega_\Delta^{k+1}-\hat\omega_\Delta^k\|_{X_\Delta}\\
			&\ge
			(1-\nu L_\Delta)
			\|\omega_\Delta^k-\hat\omega_\Delta^k\|_{X_\Delta}.
		\end{aligned}
		\]
		Dividing by \(\nu(1-\nu L_\Delta)>0\) gives
		\[
		r_\nu(\omega_\Delta^k)
		\le
		\frac{\|\omega_\Delta^{k+1}-\omega_\Delta^k\|_{X_\Delta}}
		{\nu(1-\nu L_\Delta)}.
		\]
		Accordingly, the stopping rule above guarantees
		\(r_\nu(\omega_\Delta^k)
		<\mathrm{tol}/[\nu(1-\nu L_\Delta)]\).
		Replacing \(K_{\Delta,M}\) by \(K_{\Delta,M,g}\) in the
		iteration and in the definition of \(r_\nu\), the same estimate
		holds with \(L_\Delta\) replaced by the mesh-independent bound \(L\)
		from Proposition~\ref{prop:lip}, provided \(0<\nu L<1\).
	\end{remark}

	\begin{remark}[Exact and clipped projections]
		\label{rem:projection}
		Theorem~\ref{thm:convergence} concerns the exact projection
		\eqref{eq:exact-metric-projection}. In the experiments,
		\(\Pi_u\) in \eqref{eq:value-metric-projection} is replaced by
		componentwise clipping. 
		The clipped and exact \(K_{\Delta,M}\)-projected updates coincide
		when the value-function clipping is inactive. If the same
		value-function candidate also satisfies the slope bounds, the update
		equals the exact \(K_{\Delta,M,g}\)-projection. These conditions
		are checked in Section~\ref{sec:num-bounds}.
	\end{remark}
	
	\subsection{Convergence of the projected extragradient iteration}
	\label{sec:convergence}

	If \(\omega_\Delta^*\) lies in the relative interior of
	\(K_{\Delta,M}\), then \eqref{eq:VI} is equivalent to
	\(F_\Delta(\omega_\Delta^*)=0\), and hence to
	\(A_\Delta(\omega_\Delta^*)=0\). On the boundary, write
	\(\omega_\Delta^*=(m,u,\varpi)\). Density variations give
	\begin{equation}\label{eq:complementarity}
		R_{\Delta,\mathrm{HJB}}(\omega_\Delta^*)_{i,j}
		\begin{cases}
			\ge0, & m_{i,j+1}=0,\\
			=0, & 0<m_{i,j+1}<M_m,\\
			\le0, & m_{i,j+1}=M_m,
		\end{cases}
	\end{equation}
	so \(u\) satisfies the discrete HJB equation where
	\(0<m_{i,j+1}<M_m\) and the subsolution inequality where
	\(m_{i,j+1}=0\). If the price bound is
	inactive, price variations give \(R_{\Delta,\mathrm{mc}}(\omega_\Delta^*)=0\).
	
	If the value-function box is inactive at a solution of
	\eqref{eq:VI}, both signs of sufficiently small value-function
	variations are feasible, so \(R_{\Delta,\mathrm{tr}}=0\). Combining
	\eqref{eq:disc-Rm} with the identity~\eqref{eq:Lstar-sum-zero} gives
	\[
	\Dx\sum_{i=0}^{I}m_{i,j}=1,
	\qquad j=1,\dots,J.
	\]
	Thus mass is conserved at such a solution, although it is not imposed
	during the iteration. For the VI on \(K_{\Delta,M,g}\), the same
	argument gives \(R_{\Delta,\mathrm{tr}}=0\) if the slope constraints
	are also inactive. Mass conservation still follows from inactivity
	of the value-function box alone, by taking spatially constant
	variations at each free time level.
	
	\begin{theorem}[Convergence]
		\label{thm:convergence}
		Let \((\cdot,\cdot)_{X_\Delta}\) be an inner product on
		\(X_\Delta\), with associated Riesz map \(\mathcal R_\Delta\), and set
		\(F_\Delta=\mathcal R_\Delta A_\Delta\). 
		Assume that
		\(K_{\Delta,M}\ne\varnothing\) and that \(F_\Delta\) is monotone and
		\(L_\Delta\)-Lipschitz on \(K_{\Delta,M}\), for some
		\(L_\Delta>0\), with respect to \(\|\cdot\|_{X_\Delta}\).
		 Let the step size satisfy
		\begin{equation}\label{eq:step-size}
			0<\nu<\frac{1}{L_\Delta}.
		\end{equation}
		Then \eqref{eq:VI} has a solution and, for every
		\(\omega_\Delta^0\in K_{\Delta,M}\), the iteration
		\eqref{eq:eg-pred}--\eqref{eq:eg-corr} is well defined, remains in
		\(K_{\Delta,M}\), and converges in
		\(\|\cdot\|_{X_\Delta}\) to a solution of \eqref{eq:VI}.
	\end{theorem}
	\begin{proof}
		Since \(K_{\Delta,M}\) is nonempty, compact, and convex and
		\(F_\Delta\) is continuous, \eqref{eq:VI} has a solution by the
		finite-dimensional Hartman--Stampacchia theorem; see
		\cite[Vol.~I, Section~2.2]{facchinei2003finite}. The metric projection is well defined, so
		all predictor and corrector iterates remain in \(K_{\Delta,M}\).
		
		Fix a solution
		\(\omega_\Delta^{*}\in K_{\Delta,M}\) of \eqref{eq:VI}.
		By the
		characterization of the metric projection, the predictor step gives
		\begin{equation}\label{eq:proj-predict}
			\bigl(
			\hat\omega_\Delta^{k}-\omega_\Delta^{k}
			+\nu\,F_\Delta(\omega_\Delta^{k}),
			\tilde\omega_\Delta-\hat\omega_\Delta^{k}
			\bigr)_{X_\Delta}
			\ge 0
			\qquad
			\forall\,\tilde\omega_\Delta\in K_{\Delta,M},
		\end{equation}
		and the corrector step gives
		\begin{equation}\label{eq:proj-correct}
			\bigl(
			\omega_\Delta^{k+1}-\omega_\Delta^{k}
			+\nu\,F_\Delta(\hat\omega_\Delta^{k}),
			\tilde\omega_\Delta-\omega_\Delta^{k+1}
			\bigr)_{X_\Delta}
			\ge 0
			\qquad
			\forall\,\tilde\omega_\Delta\in K_{\Delta,M}.
		\end{equation}
		
		For brevity, set
		\[
		x=\omega_\Delta^k,\qquad
		y=\hat\omega_\Delta^k,\qquad
		z=\omega_\Delta^{k+1},\qquad
		s=\omega_\Delta^*,
		\]
		and use the inner product and norm of \(X_\Delta\).
		
		Combining \eqref{eq:proj-correct} for \(\tilde\omega_\Delta=s\) with the elementary identity
		\[\|z-s\|^2+\|z-x\|^2-2 (z-s, z-x)=\|x-s\|^2\] yields
		\[
		\|z-s\|^2
		\le
		\|x-s\|^2-\|z-x\|^2
		-2\nu(F_\Delta(y),z-s).
		\]
		Monotonicity and \eqref{eq:VI} give
		\[
		(F_\Delta(y),y-s)\ge(F_\Delta(s),y-s)\ge0.
		\]
		Accordingly 
		\begin{equation}
	    \label{corrineq}
		\|z-s\|^2
		\leq \|x-s\|^2-\|z-x\|^2-2\nu(F_\Delta(y),z-y).
		\end{equation}
		Using \(\tilde\omega_\Delta=z\) in 
		\eqref{eq:proj-predict}  yields
		\begin{equation}
			\label{predicineq}
		-2 \nu (F_\Delta(x),z-y)\leq 2(y-x,z-y)=-\|y-x\|^2-\|z-y\|^2+\|z-x\|^2.
		\end{equation}
        Combining \eqref{corrineq} with \eqref{predicineq} gives
		\begin{align}
			\|z-s\|^2
			&\le
			\|x-s\|^2-\|z-x\|^2
			-2\nu(F_\Delta(y),z-y)
			\notag\\
			&\le
			\|x-s\|^2-\|y-x\|^2-\|z-y\|^2
			+2\nu(F_\Delta(x)-F_\Delta(y),z-y).
			\label{eq:eg-detail}
		\end{align}		
		By Cauchy--Schwarz, \(L_\Delta\)-Lipschitz continuity on
		\(K_{\Delta,M}\), and Young's inequality,
		\[
		2\nu(F_\Delta(x)-F_\Delta(y),z-y)
		\le
		\nu^2L_\Delta^2\|y-x\|^2+\|z-y\|^2.
		\]
		Substituting into~\eqref{eq:eg-detail} and returning to the original
		notation yields the Fej\'er estimate
		\begin{equation}\label{eq:fejer}
			\|\omega_\Delta^{k+1}-\omega_\Delta^{*}\|_{X_\Delta}^2
			\le
			\|\omega_\Delta^{k}-\omega_\Delta^{*}\|_{X_\Delta}^2
			-
			\bigl(1-\nu^2L_\Delta^2\bigr)\,
			\|\hat\omega_\Delta^{k}-\omega_\Delta^{k}\|_{X_\Delta}^2.
		\end{equation}
		
		Since $\nu<1/L_\Delta$, the coefficient $1-\nu^2L_\Delta^2$ is positive.
		Therefore
		\[
		\sum_{k=0}^\infty
		\|\hat\omega_\Delta^{k}-\omega_\Delta^{k}\|_{X_\Delta}^2<\infty,
		\]
		so
		\[
		\|\hat\omega_\Delta^{k}-\omega_\Delta^{k}\|_{X_\Delta}\to 0.
		\]
		
		Compactness of \(K_{\Delta,M}\) gives a subsequence
		\(\omega_\Delta^{k_n}\to\bar\omega\). Since
		\(\|\hat\omega_\Delta^{k}-\omega_\Delta^{k}\|_{X_\Delta}\to0\), also
		\(\hat\omega_\Delta^{k_n}\to\bar\omega\). Passing to the limit in
		\eqref{eq:proj-predict} gives
		\[
		(F_\Delta(\bar\omega),\tilde\omega_\Delta-\bar\omega)_{X_\Delta}
		\ge0
		\qquad
		\forall\,\tilde\omega_\Delta\in K_{\Delta,M},
		\]
		so \(\bar\omega\) solves \eqref{eq:VI}. Applying
		\eqref{eq:fejer} with \(\bar\omega\) as reference, the sequence
		\(\|\omega_\Delta^k-\bar\omega\|_{X_\Delta}\) is nonincreasing and
		tends to zero along the subsequence. Hence the full sequence converges
		to \(\bar\omega\).
	\end{proof}
	
	The same result applies to any nonempty compact convex feasible set,
	in particular to \(K_{\Delta,M,g}\).
	
	\begin{corollary}[Mesh-independent step size]
		\label{cor:mesh-independent-step}
		Equip \(X_\Delta\) with the inner product~\eqref{eq:disc-ip}, and let
		\(L\) be the mesh-independent constant in Proposition~\ref{prop:lip}. If \(K_{\Delta,M,g}\ne\varnothing\), then
		the iteration \eqref{eq:eg-pred}--\eqref{eq:eg-corr}, with
		\(K_{\Delta,M,g}\) in place of \(K_{\Delta,M}\), converges from every
		\(\omega_\Delta^0\in K_{\Delta,M,g}\) to a solution of the corresponding
		VI, on each fixed grid and for every fixed step size
		\[
		0<\nu<\frac1L.
		\]
	\end{corollary}
	
	For the \(K_{\Delta,M}\)-projected method, the same convergence
	conclusion holds with \(0<\nu<1/L\) if all states of the infinite
	iteration, including predictors, remain in \(K_{\Delta,M,g}\).
	In the experiments, the checks of Section~\ref{sec:num-bounds},
	together with Remark~\ref{rem:projection}, identify the computed
	trajectory with a finite initial segment of the exact
	\(K_{\Delta,M,g}\)-projected method.

	\section{Quadratic Analytical Benchmark}
	\label{sec:benchmark}

	For the benchmark, we use the quadratic running cost
	\[
	\ell(t,\alpha)=\frac{\alpha^2}{2}+\alpha\varpi(t).
	\]
	Under the dynamics \(\dot X_t=\alpha_t\), the HJB equation uses
	\[
	H(x,p)=\frac{p^2}{2},
	\qquad
	p=\varpi+u_x,
	\]
	with optimal feedback \(\alpha^*=-p\). The first-order
	MFG system is
	\begin{equation}\label{eq:quad-system}
		\begin{cases}
			-u_t + \frac{(\varpi+u_x)^2}{2} = 0, \\[4pt]
			m_t - \bigl(m(\varpi+u_x)\bigr)_x = 0, \\[4pt]
			\displaystyle\int_\R m(\varpi+u_x)\,dx = -Q(t),
		\end{cases}
	\end{equation}
	with $m(x,0)=\bar m(x)$ and $u(x,T) = \bar u(x)$.
	
	To validate the numerical method, we choose a quadratic terminal cost
	for which \eqref{eq:quad-system} reduces to a closed system of ODEs and
	admits an explicit solution; see
	\cite{gomesMeanFieldGameApproach2021}.
	
	\subsection{Benchmark data}
	
	The terminal cost is
	\begin{equation}\label{eq:terminal}
		\bar u(x) = \frac{\gamma}{2}(x-\xi)^2,
	\end{equation}
	where $\gamma > 0$ and $\xi\in\R$ are fixed parameters. 
	This cost is not covered by the whole-line assumption
	\ref{H:potential}; the explicit whole-line solution is derived directly
	and then restricted to the fixed computational interval.
	The initial density is the normalized smooth bump
	\begin{equation}\label{eq:m0}
		\bar m(x) = \frac{\hat m(x)}{\int_{-1}^1\hat m(y)\,dy}, \qquad
		\hat m(x) = \begin{cases}
			\exp\bigl(-\frac{1}{1-x^2}\bigr), & |x|<1, \\
			0, & |x|\ge 1.
		\end{cases}
	\end{equation}
	
	The supply \(Q\) is defined by
	\begin{equation}\label{eq:supply}
		\dot Q(t) = \bar Q(t) - \kappa\,Q(t),
		\qquad Q(0) = s_0,
	\end{equation}
	with the supply forcing $\bar Q(t) = 5\sin(3\pi t)$, mean-reversion rate $\kappa = 4$,
	and initial value $s_0 = -0.5$. 
	
	\subsection{The exact solution}
	
	We seek a quadratic value function of the form
	\begin{equation}\label{eq:ansatz}
		u(x,t) = \frac{A(t)}{2}\,x^2 + B(t)\,x + C(t).
	\end{equation}
	Substitution into the HJB equation yields
	\begin{equation}\label{eq:ode-system}
		\begin{cases}
			\dot A(t) = A(t)^2, \\
			\dot B(t) = A(t)(\varpi(t)+B(t)), \\
			\dot C(t) = \frac{1}{2}(\varpi(t)+B(t))^2,
		\end{cases}
	\end{equation}
	with terminal conditions $A(T)=\gamma$, $B(T)=-\gamma\xi$,
	$C(T)=\frac{\gamma}{2}\xi^2$.
	
	The Riccati equation has the solution
	\begin{equation}\label{eq:Riccati}
		A(t) = \frac{\gamma}{1+\gamma(T-t)}.
	\end{equation}
	
	Let
	\[
	\bar x(t)=\int_\R x\,m(x,t)\,dx
	\]
	denote the mean of the density. Since the transport equation preserves
	the unit mass of \(m\), the market-clearing condition gives
	\begin{equation}\label{eq:price-formula}
		\varpi(t) = -B(t) - A(t)\bar x(t) - Q(t).
	\end{equation}
	Therefore, the $B$-equation becomes
	\[
	\dot B(t)=A(t) (-A(t)\bar x(t) -Q(t)).
	\]
	For the compactly supported density constructed below, integration
	by parts in the transport equation gives
	\[
	\dot{\bar x}(t)
	=
	-\int_\R m(x,t)\bigl(\varpi(t)+u_x(x,t)\bigr)\,dx
	=
	Q(t).
	\]
	Since
	\(\bar m\) is symmetric, \(\bar x(0)=0\), and hence
	\begin{equation}\label{eq:xbar}
		\bar x(t) = \int_0^t Q(s)\,ds.
	\end{equation}	
	Since \(\dot A=A^2\) and \(\dot{\bar x}=Q\), substituting
	\eqref{eq:price-formula} into the \(B\)-equation gives
\[
\frac{d}{dt}\bigl(B+A\bar x\bigr)
=
A(-A\bar x-Q)+A^2\bar x+AQ
=
0.
\]
	Therefore, taking into account the terminal conditions, 
	\begin{equation}\label{eq:B-formula}
		B(t)
		=
		-\gamma\xi+\gamma\bar x(T)-A(t)\bar x(t),
	\end{equation}
	and \eqref{eq:price-formula} becomes
	\begin{equation}\label{eq:exact-price}
		\varpi(t)
		=
		\gamma\bigl(\xi-\bar x(T)\bigr)-Q(t).
	\end{equation}
	
	Integrating the \(C\)-equation backward from \(T\) gives
	\begin{equation}\label{eq:C-formula}
		C(t) = \frac{\gamma}{2}\xi^2
		- \frac{1}{2}\int_t^T
		\bigl(A(s)\bar x(s) + Q(s)\bigr)^2\,ds.
	\end{equation}
	
	The transport equation can be written as
	\begin{equation}
		m_t + \partial_x\bigl(v(x,t)\,m\bigr) = 0,
		\qquad v(x,t) = Q(t) - A(t)(x-\bar x(t)).
	\end{equation}
	For a characteristic with \(X(0)=x_0\),
	\[
	\frac{d}{dt}\bigl(X(t)-\bar x(t)\bigr)
	=
	-A(t)\bigl(X(t)-\bar x(t)\bigr).
	\]
	Define
	\begin{equation}\label{eq:Lambda}
		\Lambda(t)
		=
		\exp\!\left(-\int_0^t A(s)\,ds\right)
		=
		\frac{1+\gamma(T-t)}{1+\gamma T}.
	\end{equation}
	The characteristic equation gives
	\begin{equation}\label{eq:char}
		X(t;x_0) = \bar x(t) + \Lambda(t)\,x_0.
	\end{equation}
	Since \(\Lambda(t)>0\), the pushforward of \(\bar m\) under this affine
	map is
	\begin{equation}\label{eq:exact-density}
		m(x,t) = \frac{1}{\Lambda(t)}\,
		\bar m\!\Bigl(\frac{x-\bar x(t)}{\Lambda(t)}\Bigr).
	\end{equation}
	In particular,
	\[
	\operatorname{supp}m(\cdot,t)
	=
	\bar x(t)+\Lambda(t)[-1,1],
	\]
	so the density remains compactly supported.
	These formulas solve \eqref{eq:quad-system} and provide the analytical
	benchmark for \(u\), \(m\), and \(\varpi\).
	

	\section{Numerical Experiments}
	\label{sec:numerics}
	
	We report experiments for the discrete method of
	Sections~\ref{sec:discrete}--\ref{sec:eg}. 
	
	\subsection{Setup}
	\label{sec:num-setup}

	The benchmark is that of Section~\ref{sec:benchmark}: \(T=1\),
	\(\gamma=1\), \(\xi=0.25\), \([a,b]=[-1.5,1.5]\), initial
	density~\eqref{eq:m0}, terminal cost
	\(\bar u(x)=\tfrac12(x-\tfrac14)^2\), and supply
	equation~\eqref{eq:supply}.
	The resulting supply satisfies
	\[
	-\frac12\le Q(t)<\frac12,
	\qquad 0\le t\le1.
	\]
	Since \(|\dot{\bar x}|=|Q|\le\tfrac12\) and \(\Lambda(t)=1-t/2\), we have \(|\bar x(t)|\le t/2\) and
	\[
	\operatorname{supp}m(\cdot,t)
	=
	[\bar x(t)-\Lambda(t),\,\bar x(t)+\Lambda(t)]
	\subset[-1,1],
	\qquad 0\le t\le1,
	\]
	well inside \([a,b]=[-1.5,1.5]\).
	
	\emph{Discretization.} We use \(J=I/2\), so
	\(\Dt=\tfrac23\Dx\), with the degrees of freedom and residuals defined
	in Sections~\ref{sec:grid}--\ref{sec:disc-residuals}. The initial
	density is sampled and renormalized:
	\[
	\bar m_i
	=
	\frac{\bar m(x_i)}
	{\Dx\sum_k\bar m(x_k)},
	\qquad
	\Dx\sum_i\bar m_i=1.
	\]
	The scheme uses the right-endpoint density \(m_{\cdot,j+1}\) and the
	reduced boundary Hamiltonians. 
	At the sampled exact solution, the interior HJB residual decreases at
	first order; for example,
	\[
	\max_{\substack{1\le i\le I-1\\0\le j\le J-1}}
	|R_{\Delta,\mathrm{HJB},i,j}|
	\approx
	0.158,\ 0.082,\ 0.041
	\]
	for \(I=60,120,240\).
	
	\emph{Solver.} We use Algorithm~1 and the Riesz representative
	\eqref{eq:RhAh-def}. The value-function Riesz solve exploits the
	Kronecker structure of
	\(\Ltt{T}+\Lxx\). The spatial eigendecomposition and temporal
	factorizations are computed once per grid. Each application uses
	two dense spatial transforms and prefactorized tridiagonal solves
	in time, at cost \(O(I^2J)\).
	At each predictor and corrector step,
	the implementation applies
	\[
	\Pi_\Delta^{\mathrm{clip}}(z_\Delta)
	=
	\bigl(
	\operatorname{clip}_{[0,M_m]}(z_m),\,
	\operatorname{clip}_{[-M_u,M_u]}(z_u),\,
	\operatorname{clip}_{[-M_\varpi,M_\varpi]}(z_\varpi)
	\bigr),
	\]
	with the grid-independent bounds
	\[
	M_m=M_u=M_\varpi=5.
	\]
	Density and price clipping are exact metric projections, whereas the
	value-function metric projection
	\eqref{eq:value-metric-projection} is replaced by Euclidean clipping.
	
	\emph{Iteration parameters.} We use \(\nu=0.01\),
	\(k_{\max}=2\times10^6\), and the stopping rule
	\[
	\|\omega_\Delta^{k+1}-\omega_\Delta^k\|_{X_\Delta}
	<
	\mathrm{tol}_X
	=
	2\times10^{-6}.
	\]
	This tolerance controls the iterate increment, not the equation
	residuals. No run reached the iteration cap.
	The effect of tightening the
	tolerance is reported in Section~\ref{sec:num-mass}.

	\emph{Initialization.} On the \(60\times30\) grid, we use
	\[
	u_{i,j}=\bar u(x_i),
	\qquad
	m_{i,j}=\bar m_i,
	\qquad
	\varpi_j=0.
	\]
	Finer grids are warm-started by bilinear interpolation of \(u\) and
	\(m\) and linear interpolation of \(\varpi\), 
	followed by resetting the fixed rows
	\(u_{\cdot,J}=\bar u\) and \(m_{\cdot,0}=\bar m\), clipping the
	free degrees of freedom, renormalizing each free density time slice,
	and clipping the density once more.
	Price values beyond the
	last coarse time level are obtained by extrapolating the final coarse
	segment.

	\emph{Environment.} Computations used Wolfram Language~14.3 with a
	C-compiled residual kernel on an Apple M3~Max. The supply \(Q\) was evaluated in closed form, while
	\(\bar x\), \(B\), and \(C\) were computed by ODE integration with
	ten-digit precision and accuracy goals. Reported
	runtimes include only the solver loop and exclude eigendecomposition,
	compilation, and warm-start interpolation. 
	
	\subsection{Refinement sweep}
	\label{sec:num-iterations}
	
	Table~\ref{tab:convergence} reports the refinement sweep. For
	\(\phi\in\{m,\varpi\}\), the relative error is
	\[
	E_\phi
	=
	\frac{\|\phi_\Delta-\phi_{\mathrm{ex}}\|_{\ell^\infty}}
	{\|\phi_{\mathrm{ex}}\|_{\ell^\infty}},
	\]
	evaluated over the corresponding solved degrees of freedom.
	
	For the value function, by~\eqref{eq:complementarity} the variational inequality enforces the
	discrete HJB equation where \(0<m<M_m\), but only the subsolution
	inequality where \(m=0\). Thus \(u_\Delta\) need not agree with
	\(u_{\mathrm{ex}}\) in vacuum regions; the full-domain relative error
	remains approximately \(0.2\), attained at the spatial boundary. We
	therefore report
	\[
	E_u^{(\tau)}
	=
	\frac{\max_{m_{\mathrm{ex}}\ge\tau}|u_\Delta-u_{\mathrm{ex}}|}
	{\max_{m_{\mathrm{ex}}\ge\tau}|u_{\mathrm{ex}}|},
	\qquad
	\tau=10^{-2},
	\]
	where both \(m_{\mathrm{ex}}\) and \(u_{\mathrm{ex}}\) are evaluated at
	the value-function nodes \((x_i,t_j)\), \(0\le j\le J-1\), and the
	maxima run over the nodes with \(m_{\mathrm{ex}}\ge\tau\). With the
	threshold \(10^{-4}\), the error plateaus near
	\(4\times10^{-2}\) in a thin layer around the numerically smeared
	support front; away from this layer it decreases at first order.	These experiments do not establish uniform value-function
	convergence throughout the positive-density region.
	
	\begin{table}[htbp]
		\centering
		\caption{Refinement sweep on \([-1.5,1.5]\) with
			\(\nu=0.01\) and
			\(\mathrm{tol}_X=2\times10^{-6}\).
			Errors are relative \(\ell^\infty\) errors; the value-function error
			is restricted to \(m_{\mathrm{ex}}\ge10^{-2}\).
			The mass error is
			\(\max_{1\le j\le J}|\Dx\sum_i m_{i,j}-1|\).
			All grids after the first are warm-started, and runtimes exclude setup (Section~\ref{sec:num-setup}).}
		\label{tab:convergence}
		\renewcommand{\arraystretch}{1.2}
		\setlength{\tabcolsep}{4.5pt}
		\begin{tabular}{lccccccc}
			\toprule
			Grid & \multirow{2}{*}{$\Dx$} & \multirow{2}{*}{Iterations}
			& Runtime
			& \multicolumn{3}{c}{Relative error}
			& \multicolumn{1}{c}{Mass} \\
			\cmidrule(lr){5-7} \cmidrule(lr){8-8}
			$(I\times J)$ & & & (s) & $\varpi$ & $m$ & \(u|_{m_{\mathrm{ex}}\ge10^{-2}}\)& error \\
			\midrule
			$60\times 30$
			& 0.0500   & 61\,765  & 25
			& $6.48\!\times\!10^{-2}$ & $1.79\!\times\!10^{-1}$
			& $4.77\!\times\!10^{-2}$ & $4.1\!\times\!10^{-4}$ \\
			$120\times 60$
			& 0.0250   & 39\,995  & 44
			& $3.22\!\times\!10^{-2}$ & $1.25\!\times\!10^{-1}$
			& $2.42\!\times\!10^{-2}$ & $1.9\!\times\!10^{-4}$ \\
			$240\times 120$
			& 0.0125   & 29\,601  & 89
			& $1.60\!\times\!10^{-2}$ & $8.48\!\times\!10^{-2}$
			& $1.21\!\times\!10^{-2}$ & $1.6\!\times\!10^{-4}$ \\
			$480\times 240$
			& 0.00625  & 20\,846  & 206
			& $8.01\!\times\!10^{-3}$ & $5.41\!\times\!10^{-2}$
			& $6.07\!\times\!10^{-3}$ & $1.8\!\times\!10^{-4}$ \\
			$960\times 480$
			& 0.003125 & 12\,928  & 471
			& $4.00\!\times\!10^{-3}$ & $3.30\!\times\!10^{-2}$
			& $3.02\!\times\!10^{-3}$ & $1.7\!\times\!10^{-4}$ \\
			\bottomrule
		\end{tabular}
	\end{table}

	\emph{Convergence rates.}
	The price and bulk value-function errors exhibit approximately
	first-order decay under refinement.
	The density
	error has empirical \(\ell^\infty\) order increasing from about \(0.5\) to \(0.7\),
	with its maximum concentrated near the smeared support front. 
	
	\emph{Step size and iteration counts.}
	The same step size \(\nu=0.01\) is used on every grid.  For
	\(M_m=M_g=M_\varpi=5\) and \(|\Omega|=3\), the explicit
	constant~\eqref{eq:L-explicit} gives \(L\le\sqrt{6902}<100\), so
	\(\nu=0.01\) satisfies the mesh-independent step-size condition on
	\(K_{\Delta,M,g}\). Together with the projection verification in
	Section~\ref{sec:num-bounds}, this identifies each computed trajectory
	with an initial segment of the exact iteration of
	Corollary~\ref{cor:mesh-independent-step}.
	The iteration count decreases from \(61{,}765\) on the cold-started
	coarsest grid to \(12{,}928\) on the warm-started finest grid.
	Since the initialization changes with refinement, these counts
	do not isolate the effect of the mesh size.
	
	\subsection{Mass conservation}
	\label{sec:num-mass}
	
	Mass is not enforced during the iteration. Let
	\[
	M_j=\Dx\sum_i m_{i,j},
	\qquad
	M_0=1.
	\]
	Since \(\Dx\sum_i\bar m_i=1\), summing \eqref{eq:disc-Rm} over space
	and using the identity~\eqref{eq:Lstar-sum-zero} gives
	\[
	M_{j+1}-M_j
	=
	\Dt\,\Dx\sum_i
	R_{\Delta,\mathrm{tr}}(\omega_\Delta)_{i,j},
	\qquad
	j=0,\dots,J-1.
	\]
	Consequently,
	\[
	\max_{1\le j\le J}|M_j-1|
	\le
	\Dt\sum_{n=0}^{J-1}
	\left|
	\Dx\sum_i
	R_{\Delta,\mathrm{tr}}(\omega_\Delta)_{i,n}
	\right|.
	\]

	At \(\mathrm{tol}_X=2\times10^{-6}\), the mass errors range from
	\(1.6\times10^{-4}\) to \(4.1\times10^{-4}\). On the
	\(240\times120\) grid, reducing the tolerance to
	\(2\times10^{-7}\) lowers the mass error from
	\(1.6\times10^{-4}\) to \(1.7\times10^{-5}\), 
	while leaving the price and density errors and \(E_u^{(10^{-2})}\)
	unchanged to three significant digits.
	Here the mass error is the largest deviation of a time-slice
	mass from \(1\). Continuing the iteration on the same grid
	reduces this deviation by about a factor of nine.
	The unchanged solution errors suggest that discretization
	dominates these errors at this resolution.

	\subsection{Monitored bounds and state-constraint structure}
	\label{sec:num-bounds}
	
	\emph{A posteriori verification of the \(K_{\Delta,M,g}\) bounds.}
	The bounds \(M_m=M_u=M_\varpi=5\) are enforced by clipping,
	while the slope threshold \(M_g=5\) is monitored for every initial,
	predictor, and corrector state. Across all five runs,
	\[
	\max m\le1.85,
	\qquad
	\max\{|D_+u|,|D_-u|\}\le1.73,
	\qquad
	\max|\varpi|\le0.81.
	\]
	The upper density bound and the value-function and price clipping were
	inactive. 
	Thus every computed state belongs to \(K_{\Delta,M,g}\), with the
	same bounds on every grid. Since the value-function
	clipping was inactive and the monitored slopes stay below \(M_g\),
	the unprojected value-function updates are feasible for
	\(K_{\Delta,M,g}\), and the computed projections coincide with the
	exact \(X_\Delta\)-metric projections onto
	\(K_{\Delta,M,g}\). Figure~\ref{fig:slope_history} shows the monitored
	slopes.
	
	\begin{figure}[htbp]
		\centering
		\includegraphics[scale=0.7]{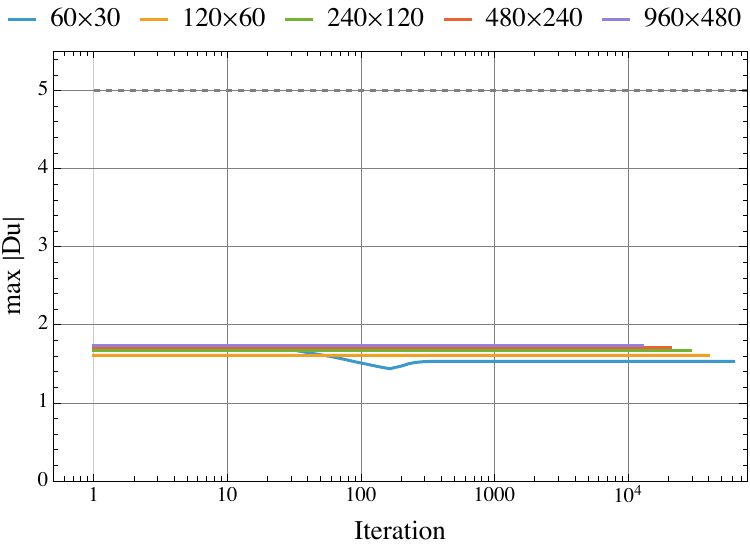}
		\caption{Maximum spatial slope of the corrector iterates on
			each grid, shown at selected iteration numbers. Predictor
			and corrector slopes were checked at every iteration.
			The iteration axis is logarithmic, and the dashed line
			marks \(M_g=5\).}
		\label{fig:slope_history}
	\end{figure}
	
	\emph{Subsolution property and state-constraint comparison.}
	The reduced boundary Hamiltonians of Section~\ref{sec:ham} provide the
	state-constraint closure of the HJB scheme. 
	At an exact VI solution with inactive upper density bound,
	\eqref{eq:complementarity} gives
	\(R_{\Delta,\mathrm{HJB}}\ge0\) everywhere, with equality
	wherever \(m_{i,j+1}>0\). With our sign convention, the first
	inequality is the discrete HJB subsolution condition.
	For the final computed iterates, we measure violations of this
	inequality through the minimum residual, and violations of HJB
	equality through the maximum absolute residual where
	\(m_{i,j+1}>10^{-12}\).
	 At \(240\times120\) and \(960\times480\),
	\(\min R_{\Delta,\mathrm{HJB}}\approx-1.4\times10^{-3}\) on both grids, while the maximum
	absolute residual on
	\(\{(i,j):m_{i,j+1}>10^{-12}\}\) is
	\(1.4\times10^{-3}\) and \(2.4\times10^{-3}\), respectively.
	
	To obtain an upper comparison bound for \(u\), let
	\(u_{\mathrm{sc}}\) denote the exact solution of the discrete
	HJB equations at the final computed price, with terminal data
	\(\bar u\) and the same reduced boundary Hamiltonians.
	This auxiliary problem imposes HJB equality at every spatial
	node, including where the computed density vanishes.
	
	Set \(\varepsilon=\max\{0,-\min R_{\Delta,\mathrm{HJB}}\}\),
	where the minimum includes the boundary nodes.
	Subtracting \(\varepsilon(T-t_j)\) from \(u_{i,j}\) increases
	the HJB residual by \(\varepsilon\), without changing the
	spatial differences or terminal data. The shifted function
	is therefore a discrete HJB subsolution, and comparison gives
	\[
	u_{i,j}\le(u_{\mathrm{sc}})_{i,j}+\varepsilon(T-t_j).
	\]
	Here \(\varepsilon\approx1.4\times10^{-3}\) in both reported runs.
	We approximate \(u_{\mathrm{sc}}\) by a backward fixed-point
	iteration with stopping tolerance \(10^{-13}\) at each time level.
	The resulting estimates of the maximum excess of \(u\) over
	\(u_{\mathrm{sc}}\) are
	\[
	\max(u-u_{\mathrm{sc}})
	\approx
	1.9\times10^{-5}
	\quad\text{and}\quad
	4.6\times10^{-5}
	\]
	on the \(240\times120\) and \(960\times480\) grids, respectively.
	For comparison with the sampled whole-line benchmark, the maximum
	of \(u_{\mathrm{sc}}-u_{\mathrm{ex}}\), evaluated at each boundary
	node and its two adjacent interior nodes over all time levels,
	is estimated to decrease from \(9\times10^{-3}\) to
	\(2.3\times10^{-3}\) between the two grids.
	This comparison includes both HJB discretization error and
	the difference between the computed and exact prices.
	The upper bound above compares \(u\) with \(u_{\mathrm{sc}}\)
	at the same computed price; it does not imply an exact ordering
	between \(u\) and \(u_{\mathrm{ex}}\) on a fixed grid.
	
	\subsection{Solution fields and refinement errors}
	\label{sec:num-figures}
	
	Figures~\ref{fig:density_compare_rev}--\ref{fig:price_compare_rev}
	compare the exact and numerical solutions on the
	\(960\times480\) grid, and Figure~\ref{fig:conv_refine_rev}
	summarizes the refinement errors.
	
	\begin{figure}[htbp]
		{\includegraphics[width=0.48\textwidth]{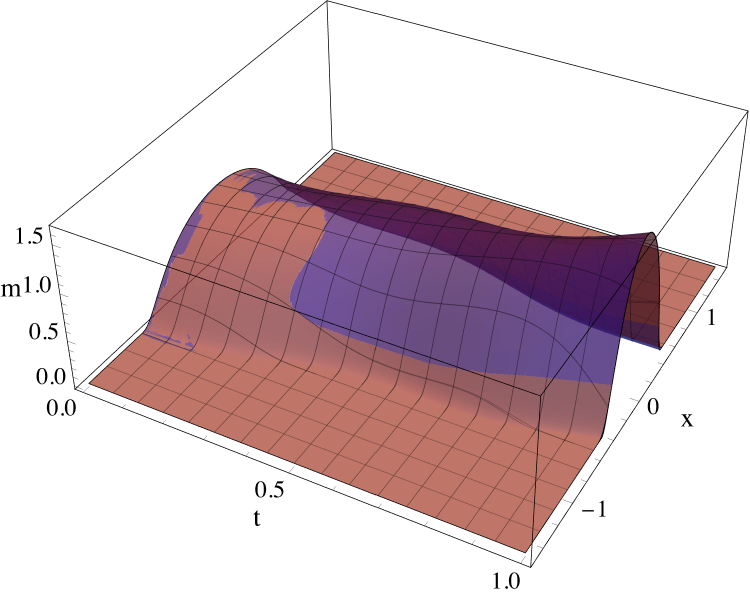}}
		{\includegraphics[width=0.5\textwidth]{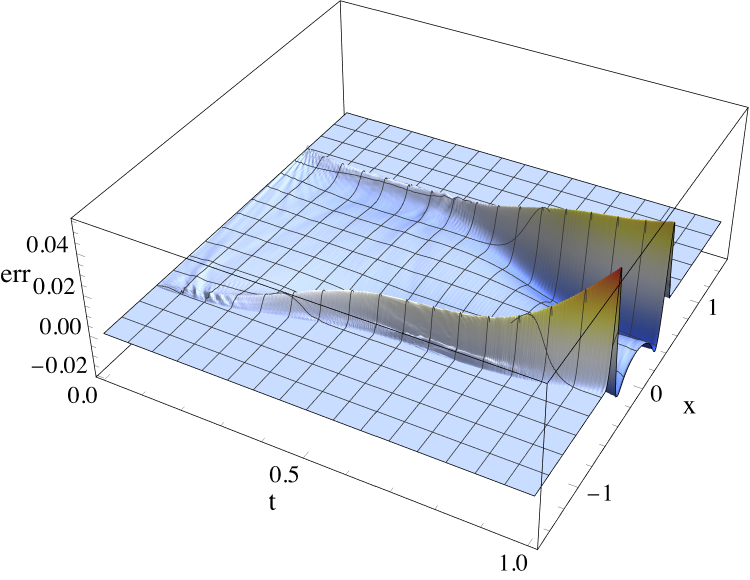}}
		\caption{Exact (blue) and numerical (orange) densities $m(x,t)$ on the
			$960\times480$ grid (left) and the corresponding signed error field \(m_\Delta-m_{\mathrm{ex}}\)
			(right).}
		\label{fig:density_compare_rev}
	\end{figure}
	
	\begin{figure}[htbp]
		{\includegraphics[width=0.48\textwidth]{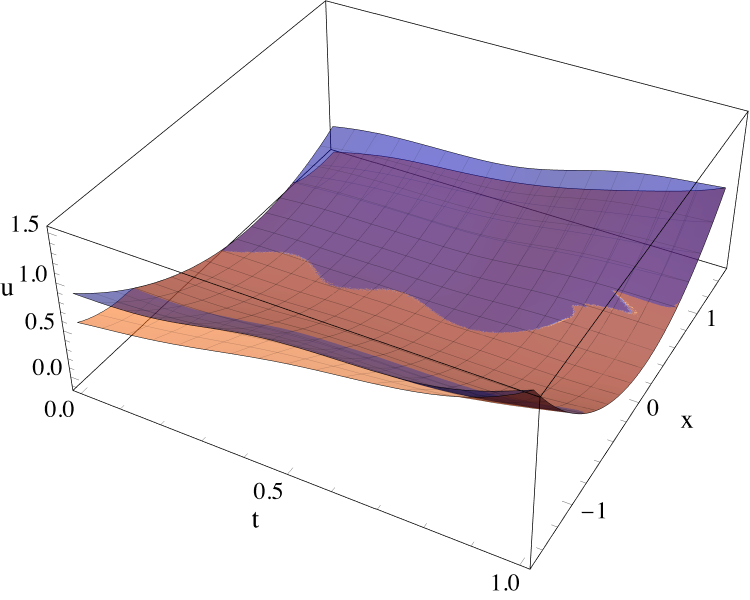}}
		{\includegraphics[width=0.5\textwidth]{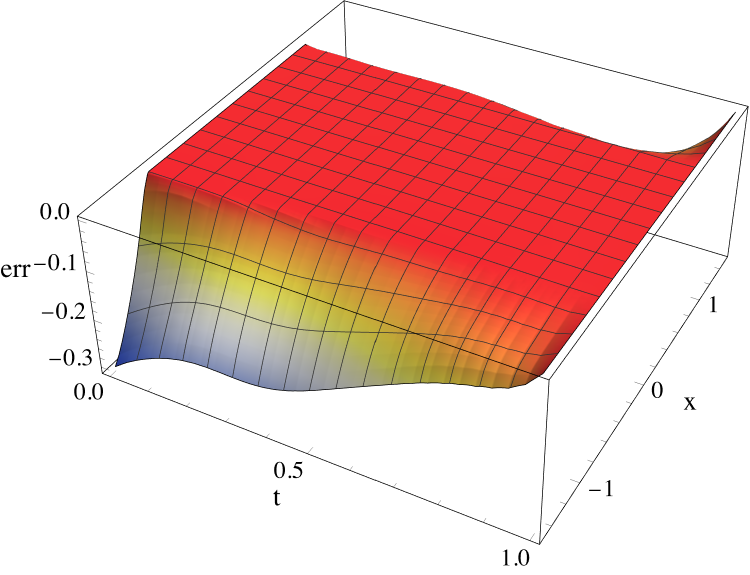}}
		\caption{Exact (blue) and numerical (orange) value functions $u(x,t)$ on the
			$960\times480$ grid (left) and the corresponding signed error field \(u_\Delta-u_{\mathrm{ex}}\)
			(right).
			The error is concentrated in vacuum regions, where the variational
			inequality does not impose HJB equality.}
		\label{fig:value_compare_rev}
	\end{figure}
	
	\begin{figure}[htbp]
		{\includegraphics[width=0.485\textwidth]{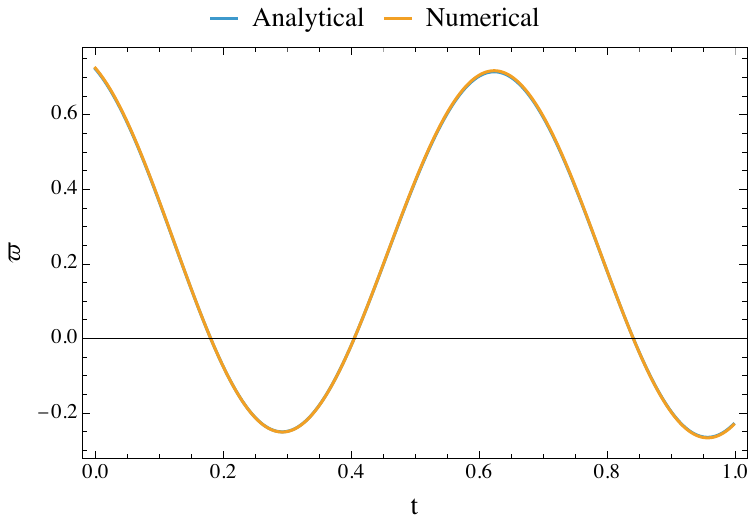}}
		{\includegraphics[width=0.5\textwidth]{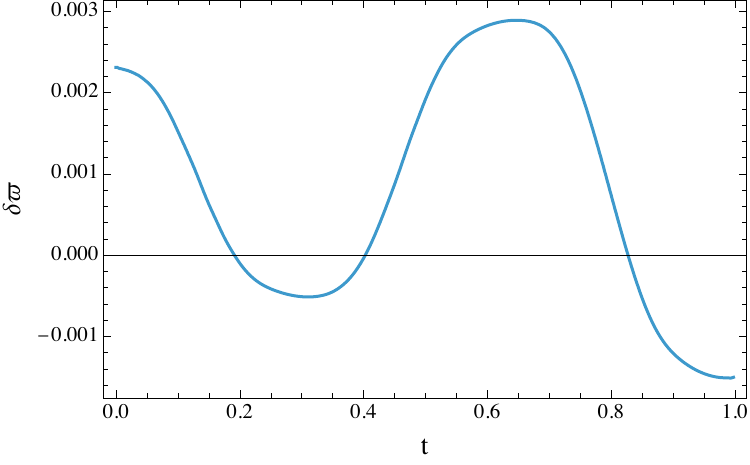}}
		\caption{Exact and numerical price trajectories $\varpi(t)$ on the
			$960\times480$ grid (left) and the corresponding signed error profile \(\varpi_\Delta-\varpi_{\mathrm{ex}}\)
			(right).}
		\label{fig:price_compare_rev}
	\end{figure}
	
	\begin{figure}[htbp]
		\centering
		\includegraphics[scale=0.7]{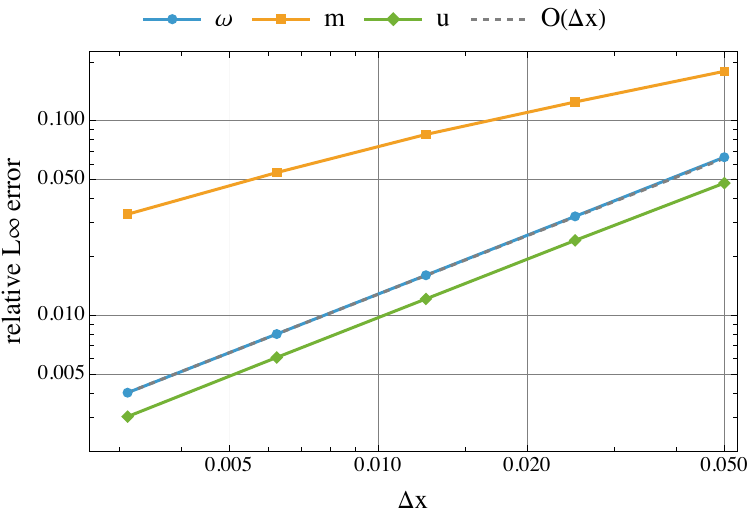}
		\caption{Relative \(\ell^\infty\) errors versus
			\(\Dx\), with an \(O(\Dx)\) reference slope for the price and
			value-function errors. The value-function
			error is restricted to \(m_{\mathrm{ex}}\ge10^{-2}\).}
		\label{fig:conv_refine_rev}
	\end{figure}
	
	\section{Conclusion and Future Work}
	\label{sec:conclusion}

	The weak MFG formulation determines the block structure of a natural
	Riesz preconditioner: the HJB and market-clearing residuals are
	represented in discrete \(L^2\) for the density and price, while the
	transport residual is lifted through the space-time \(H^1\)
	value-function norm. For the quadratic monotone upwind discretization, the
	resulting Riesz representative is monotone and, under mesh-independent
	density, slope, and price bounds, uniformly Lipschitz. The
	extragradient theorem then gives convergence of the exact
	metric-projection method on each fixed grid and a grid-independent
	admissible step size when the method is posed on
	\(K_{\Delta,M,g}\). The experiments are consistent with these results:
	the same step size was used throughout the refinement sweep, the
	monitored bounds remained uniform, and the price, density, and bulk
	value-function errors decreased under refinement.

	The verification in Section~\ref{sec:num-bounds} shows
	that the reported trajectories coincide with finite
	initial segments of the exact metric-projection method.
	Active value-function clipping requires a separate analysis.
	The main remaining problems are
	therefore to analyze this variant or develop an efficient exact metric
	projection, 
	derive a priori mesh-independent slope bounds for the clipped iteration,
	and establish convergence under mesh
	refinement. 
	
	\section*{Acknowledgments}
	
	The authors thank King Abdullah University of Science and Technology (KAUST), Saudi Arabia, for support.
	
	\bibliographystyle{abbrv}
	\bibliography{refs}
	
\end{document}